\documentclass[11pt,amssymb,amsfont,a4paper]{article}
\usepackage{latexsym,amssymb,amsmath,amsthm,color}
\usepackage{geometry}
\usepackage{url}
\usepackage[utf8]{inputenc}
\usepackage{authblk}

\theoremstyle{definition}
\newtheorem{definition}{Definition}
\newtheorem{example}[definition]{Example}

\theoremstyle{plain}
\newtheorem{theorem}{Theorem}
\newtheorem{proposition}[definition]{Proposition}
\newtheorem{lemma}[definition]{Lemma}
\newtheorem{remark}[definition]{Remark}
\newtheorem{corollary}[definition]{Corollary}

\title{Evaluation and interpolation over multivariate skew polynomial rings}
\author[1,2]{Umberto Mart{\'i}nez-Pe\~{n}as \thanks{umberto@ece.utoronto.ca}}
\author[1]{Frank R. Kschischang \thanks{frank@ece.utoronto.ca}}
\affil[1]{Dept.\ of Electrical \& Computer Engineering,
University of Toronto, Canada}
\affil[2]{Dept.\ of Mathematical Sciences, Aalborg University, Denmark}

\date{}

\begin{document}

\maketitle

\begin{abstract}
The concepts of evaluation and interpolation are extended from univariate skew polynomials to multivariate skew polynomials, with coefficients over division rings. Iterated skew polynomial rings are in general not suitable for this purpose. Instead, multivariate skew polynomial rings are constructed in this work as follows: First, free multivariate skew polynomial rings are defined, where multiplication is additive on degrees and restricts to concatenation for monomials. This allows to define the evaluation of any skew polynomial at any point by unique remainder division. Multivariate skew polynomial rings are then defined as the quotient of the free ring by (two-sided) ideals that vanish at every point. The main objectives and results of this work are descriptions of the sets of zeros of these multivariate skew polynomials, the families of functions that such skew polynomials define, and how to perform Lagrange interpolation with them. To obtain these descriptions, the existing concepts of P-closed sets, P-independence, P-bases (which are shown to form a matroid) and skew Vandermonde matrices are extended from the univariate case to the multivariate one.


\end{abstract}

\section{Introduction} \label{sec intro}

\textit{Univariate skew polynomial rings}, introduced in \cite{ore}, are those ``non-commutative polynomial rings'', over  some coefficient ring, whose addition is the usual one, but whose multiplication is arbitrary with the following restrictions: The $i$th power (being $ i $ a natural number) of the variable $ x $ corresponds to the monomial ``$ x^i $'', and the degree of a product of two arbitrary polynomials is the sum of their degrees. Adding the commutativity property yields the \textit{conventional} polynomial ring, that is, the monoid ring of the natural numbers over the coefficient ring. 

An extension of the concept of \textit{evaluation} to skew polynomials over division rings  was first given in \cite{lam} and further developed in \cite{algebraic-conjugacy, lam-leroy}. Since a skew polynomial ring (over a division ring) is a right-Euclidean domain \cite{ore}, the evaluation of $ F(x) $ on a point $ a $ is defined in \cite{lam,lam-leroy} as the remainder of the Euclidean division of $ F(x) $ by $ x - a $ on the right. Such extension is natural in the sense that it is based on the ``Remainder Theorem'' for conventional polynomials and it is analogous to projecting on a quotient ring defined by a maximal ideal, as in algebraic geometry. This concept of evaluation helps unify the study of Vandermonde, Moore and Wronskian matrices \cite{lam, lam-leroy} and further matrix types (see \cite[p. 604]{linearizedRS} for instance), and gives a natural framework for Hilbert 90 Theorems \cite{hilbert90} and pseudolinear transformations \cite{leroy-pol}, which unify semilinear and differential transformations (see also \cite[Sec. 8.4]{cohn}). It has also provided error-correcting codes with good minimum Hamming distance \cite{skewcyclic1}, maximum rank distance codes \cite{gabidulin}, and maximum sum-rank distance codes \cite{linearizedRS} with finite-field sizes that are not exponential in the code length, in constrast with \cite{gabidulin} (see \cite[Sec. 4.2]{linearizedRS}).

Extending this concept of evaluation to multivariate skew polynomials is not straightforward. In general, unique remainder algorithms \cite[Sec. 4]{lenstra} do not hold for \textit{iterated} skew polynomials since they do not satisfy Jategaonkar's condition \cite{jategaonkar} for $ n > 1 $ variables (see \cite[Prop. 4.7]{lenstra} and also \cite[Sec. 8.8]{cohn}). Recently in \cite{skewRM}, it is proposed to evaluate certain iterated skew polynomials (those forming Poincar{\'e}-Birkhoff-Witt extensions following \cite[Def 2.1]{zhangPBW}) at points $ (a_1, a_2, \ldots, a_n) $ where $ x_1 - a_1 $, $ x_2 - a_2 $, $ \ldots $, $ x_n - a_n $ form a G{\"o}bner-Shirshov basis, since then a unique remainder algorithm exists. However, this does not include all iterated skew polynomial rings or affine points (see \cite[Ex. 3.5]{skewRM}) and the important concepts and results from \cite{lam, algebraic-conjugacy, lam-leroy} do not seem to hold.

In this work, we overcome these issues by considering an alternative construction. We start by defining \textit{free multivariate} skew polynomial rings (using the free monoid with basis $ x_1, x_2, \ldots, x_n $)  following Ore's idea: The product of two monomials consists in appending them, and the degree of a product of two skew polynomials is the sum of their degrees. Over fields, adding commutativity between constants and variables (that is, turning the ring into an algebra) yields the conventional free algebra \cite[Sec. 0.11]{cohn} as a particular case.  Thanks to this definition, we show that we may define the evaluation of any (free) skew polynomial $ F(x_1, x_2, \ldots, x_n) $ at any affine point $ (a_1, a_2, \ldots, a_n) $ as the unique remainder of the Euclidean division of $ F(x_1, x_2, \ldots, x_n) $ by $ x_1 - a_1 $, $ x_2 - a_2 $, $ \ldots $, $ x_n - a_n $ on the right. Once this is done, we may define general (\textit{nonfree}) skew polynomial rings, where evaluation is still natural at every point, as quotients of the free ring by two-sided ideals of skew polynomials that vanish at every point (Definition \ref{def skew polynomial rings}). Reasonably behaved iterated skew polynomial rings are also quotients of the introduced free multivariate skew polynomial rings, and evaluations by unique remainder division as in \cite{skewRM} are recovered by the proposed evaluations (Remark \ref{remark iterated are also quotients}), although the converse does not seem to hold. 

Our main objective is to describe the \textit{functions} obtained by evaluating multivariate skew polynomials over division rings, under some finiteness conditions (see the paragraph below). This problem is closely related to that of interpolation in the sense of Lagrange, which has been studied previously in the univariate case in \cite{eric, lam, lam-leroy, skew-interpolation, zhang}.

Our main results are as follows: We obtain a description of the family of such functions, when defined on a finitely generated set of zeros (\textit{P-closed set}), as a left vector space over the division ring of coefficients, and we find its dimension and a left basis (Theorem \ref{th describing evaluation as vector space}). For this, we first obtain a Lagrange-type interpolation theorem (Theorem \ref{th lagrange interpolation}) on P-closed sets. To this end, we need to extend first the concept of \textit{P-independence} and \textit{P-basis} from \cite[Sec. 4]{algebraic-conjugacy}, which naturally form a matroid (Proposition \ref{prop matroid}). See \cite{oxley} for more details on matroid theory. For that purpose, we need to introduce \textit{ideals of zeros}, whose properties are based on extensions to the multivariate case of tools from \cite{lam, algebraic-conjugacy, lam-leroy}: A multiplication that is additive on degrees (Theorem \ref{th multiplication is additive}), an iterative evaluation on monomials (Theorem \ref{th fundamental functions}) and a product rule (Theorem \ref{th product rule}).

Apart from its own interest, our main motivations to develop this theory come from the theory of error-correcting codes over finite fields, in view of \cite{skewcyclic1, gabidulin, linearizedRS}, as explained above. A definition of skew Reed-Muller codes has been recently proposed \cite{skewRM}, based on evaluating certain iterated skew polynomials at certain points as noted previously. However, the core properties of skew polynomial evaluation codes rely on the matroid given by P-independence and evaluation on P-bases (see \cite{linearizedRS}), which we introduce in the multivariate case in this work. Apart from applications in coding theory, it has been recently shown in \cite{lin-multivariateskew} that Hilbert's Theorem 90 can be naturally stated and proven using the framework of this paper for general Galois extensions of fields (as considered by Noether) using arbitrary generators and relations of the Galois group (note that univariate skew polynomials restrict Hilbert 90 Theorems to a single generator \cite{hilbert90}, as originally stated by Kummer and Hilbert). A differential or more general version of such a Hilbert's Theorem 90 can be similarly put in this framework. Further applications in Galois theory or partial differential equations (such as a study of multivariate Moore or Wronskian matrices) may be possible and of interest. 

The organization is as follows. In Section \ref{sec matrix morphisms}, we show which multiplications are additive on degrees over ``free multivariate polynomial rings'' (Theorem \ref{th multiplication is additive}), extending \cite[Eq. (3), (4) \& (5)]{ore}. In Section \ref{sec evaluations}, we show how to define evaluations as remainders of Euclidean divisions and give a recursive formula for monomials (Theorem \ref{th fundamental functions}), extending \cite[Lemma 2.4]{lam-leroy} and \cite[Eq. (2.3)]{lam-leroy}. In Section \ref{sec product rule}, we show how the product of two skew polynomials is preserved after evaluation (Theorem \ref{th product rule}), extending \cite[Th. 2.7]{lam-leroy}. In Section \ref{sec zeros}, we define P-closed sets and ideals of zeros, and give their basic properties. Using them, we define in Section \ref{sec general skew} \textit{nonfree} multivariate skew polynomial rings (Definition \ref{def skew polynomial rings}). In Section \ref{sec P-bases}, we extend the crucial concepts of \textit{P-independence} and \textit{P-bases} from \cite[Sec. 4]{algebraic-conjugacy} to our context. In Section \ref{sec lagrange interpolation}, we show the existence of Lagrange interpolating skew polynomials (Theorem \ref{th lagrange interpolation}). In Section \ref{sec image and ker}, we obtain the dimension and left bases of the left vector space of skew polynomial functions over a finitely generated P-closed set (Theorem \ref{th describing evaluation as vector space}). In Section \ref{sec skew vandermonde}, we give explicit computational methods to find such dimensions and bases and to perform Lagrange interpolation, via an extension of the Vandermonde matrices considered in \cite{lam, lam-leroy}. The complexity for finding ranks and P-bases is exponential in general, but given a P-basis, the complexity of finding Lagrange interpolating polynomials is polynomial.

\section*{Notation}

Unless otherwise stated, $ \mathbb{F} $ will denote a division ring.
 Assuming $ \mathbb{F} $ to be finite (thus a field \cite{weddeburn}) avoids all other finiteness assumptions.

For positive integers $ m $ and $ n $, $ \mathbb{F}^{m \times n} $ will denote the set of $ m \times n $ matrices over $ \mathbb{F} $, and $ \mathbb{F}^n $ will denote the set of column vectors of length $ n $ over $ \mathbb{F} $. That is, $ \mathbb{F}^n = \mathbb{F}^{n \times 1} $.

On a non-commutative ring $ \mathcal{R} $, we will denote by $ (A) \subseteq \mathcal{R} $ the left ideal generated by a set $ A \subseteq \mathcal{R} $, and on a left vector space $ \mathcal{V} $ over $ \mathbb{F} $, we will denote by $ \langle B \rangle \subseteq \mathcal{V} $ the $ \mathbb{F} $-linear left vector space generated by a set $ B \subseteq \mathcal{V} $. We use the simplified notation $ (F_1, F_2, \ldots, F_n) = (\{ F_1, F_2, \ldots, F_n \}) $ and $ \langle F_1, F_2, \ldots, F_n \rangle = \langle \{ F_1, F_2, \ldots, F_n \} \rangle $.

All rings in this work will be assumed to have multiplicative identity.

\section{Free skew polynomial rings, matrix morphisms and vector derivations} \label{sec matrix morphisms}

In this section, we show which multiplications over a free non-commutative polynomial ring consist in appending monomials and are additive on degrees. See Remark \ref{remark why variables dont commute} to see why we cannot assume that variables commute with each other, unless we are dealing with conventional multivariate polynomials over fields. See Remarks \ref{remark why not iterated skew polynomials} and \ref{remark iterated are also quotients} to see why we do not consider iterated skew polynomial rings.

Fix a positive integer $ n $ from now on, let $ x_1, x_2, \ldots, x_n $ be $ n $ distinct characters, and denote by $ \mathcal{M} $ the set of all finite strings using these characters, that is, the free monoid with basis $ x_1, x_2, \ldots, x_n $ (see \cite[Sec. 6.5]{cohn}). The empy string will be denoted by $ 1 $. A character $ x_i $ will be called a \textit{variable}, an element $ \mathfrak{m} \in \mathcal{M} $ will be called a \textit{monomial}, and we will define its \textit{degree}, denoted by $ \deg(\mathfrak{m}) $, as its length as a string. 

Let $ \mathcal{R} $ be the left vector space over $ \mathbb{F} $ with basis $ \mathcal{M} $. That is, every element $ F \in \mathcal{R} $ can be expressed uniquely as a linear combination (with coefficients on the left)
$$ F = \sum_{\mathfrak{m} \in \mathcal{M}} F_\mathfrak{m} \mathfrak{m}, $$
where $ F_\mathfrak{m} \in \mathbb{F} $, for $ \mathfrak{m} \in \mathcal{M} $, and $ F_\mathfrak{m} = 0 $ except for a finite number of monomials.

An element $ F \in \mathcal{R} $ will be called a \textit{(multivariate) skew polynomial}, and we will define its \textit{degree}, denoted by $ \deg(F) $, as the maximum degree of a monomial $ \mathfrak{m} \in \mathcal{M} $ such that $ F_\mathfrak{m} \neq 0 $, if $ F \neq 0 $. We will define $ \deg(F) = \infty $ if $ F = 0 $.
 
Formally, our objective is to provide $ \mathcal{R} $ with an inner product $ \mathcal{R} \times \mathcal{R} \longrightarrow \mathcal{R} $ that turns it into a non-commutative ring with $ 1 $ as multiplicative identity, restricts to the operation $ \mathcal{M} \times \mathcal{M} \longrightarrow \mathcal{M} $ that consists in appending strings, and where the degree of a product of two skew polynomials is the sum of their degrees. 

First observe that, by identifying $ a \in \mathbb{F} $ with $ a 1 \in \mathcal{R} $, we may assume that $ \mathbb{F} \subseteq \mathcal{R} $, with the elements in $ \mathbb{F} $ called \textit{constants}. Furthermore, $ \mathbb{F} $ is a subring of $ \mathcal{R} $ as long as $ 1 $ is the multiplicative identity. Next, by inspecting constants and variables, we see that we need functions 
$$ \sigma_{i,j} : \mathbb{F} \longrightarrow \mathbb{F}, \quad \textrm{and} \quad \delta_i : \mathbb{F} \longrightarrow \mathbb{F}, $$ 
for $ i,j = 1,2, \ldots, n $, such that
\begin{equation}
x_i a = \sum_{j=1}^n \sigma_{i,j}(a) x_j + \delta_i(a),
\label{eq def inner product}
\end{equation}
for $ i = 1,2, \ldots, n $, and for all $ a \in \mathbb{F} $. This defines two maps
\begin{equation}
\sigma : \mathbb{F} \longrightarrow \mathbb{F}^{n \times n} : a \mapsto \left( \begin{array}{cccc}
\sigma_{1,1}(a) & \sigma_{1,2}(a) & \ldots & \sigma_{1,n}(a) \\
\sigma_{2,1}(a) & \sigma_{2,2}(a) & \ldots & \sigma_{2,n}(a) \\
\vdots & \vdots & \ddots & \vdots \\
\sigma_{n,1}(a) & \sigma_{n,2}(a) & \ldots & \sigma_{n,n}(a) \\
\end{array} \right),
\label{eq def matrix morphism}
\end{equation}
and
\begin{equation}
\delta : \mathbb{F} \longrightarrow \mathbb{F}^n : a \mapsto \left( \begin{array}{c}
\delta_1(a) \\
\delta_2(a) \\
\vdots \\
\delta_n(a)
\end{array} \right).
\label{eq def vector derivation}
\end{equation}
With this more compact notation, we may write Equation (\ref{eq def inner product}) as
\begin{equation}
\mathbf{x} a = \sigma(a) \mathbf{x} + \delta(a),
\label{eq def inner product compact}
\end{equation}
where $ \mathbf{x} $ is a column vector containing $ x_i $ in the $ i $th row, for $ i = 1,2, \ldots, n $. We have the following result, which extends the discussion in the case $ n = 1 $ given at the beginning of \cite{ore}. See also \cite[Th. 10.1]{cohn}.

\begin{theorem} \label{th multiplication is additive}
If an inner product in $ \mathcal{R} $ turns it into a non-commutative ring with multiplicative identity $ 1 $, consists in appending monomials when restricted to $ \mathcal{M} $ and is additive on degrees, then it is given on constants and variables as in (\ref{eq def inner product}), the map $ \sigma : \mathbb{F} \longrightarrow \mathbb{F}^{n \times n} $ in (\ref{eq def matrix morphism}) is a ring morphism, and the map $ \delta : \mathbb{F} \longrightarrow \mathbb{F}^n $ in (\ref{eq def vector derivation}) is additive and satisfies that
\begin{equation}
\delta(ab) = \sigma(a) \delta(b) + \delta(a) b,
\label{eq multiplicative property vector derivations}
\end{equation}
for all $ a,b \in \mathbb{F} $.
 
Conversely, for any two such maps $ \sigma : \mathbb{F} \longrightarrow \mathbb{F}^{n \times n} $ and $ \delta : \mathbb{F} \longrightarrow \mathbb{F}^n $, there exists a unique inner product in $ \mathcal{R} $ satisfying the properties in the previous paragraph. Furthermore, two such inner products are equal if, and only if, the corresponding maps are equal.
\end{theorem}
\begin{proof}
First assume that a given inner product in $ \mathcal{R} $ satisfies the properties given  in the first paragraph. The additive properties of $ \sigma $ and $ \delta $ then follow from
$$ x_i(a+b) = (x_ia) + (x_ib), $$
for all $ a,b \in \mathbb{F} $ and all $ i = 1,2, \ldots, n $, their multiplicative properties follow from
$$ x_i(ab) = (x_ia)b, $$
for all $ a,b \in \mathbb{F} $ and all $ i = 1,2, \ldots, n $, and $ \sigma(1) = I $ follows from $ x_i 1 = 1 x_i $ (since $ 1 $ is a multiplicative identity) for all $ i = 1,2, \ldots, n $.

Next, the uniqueness and equality properties in the second paragraph are straightforward using Equations (\ref{eq def inner product}) or (\ref{eq def inner product compact}). 

Finally, given a ring morphism $ \sigma : \mathbb{F} \longrightarrow \mathbb{F}^{n \times n} $ and an additive map $ \delta : \mathbb{F} \longrightarrow \mathbb{F}^n $ satisfying (\ref{eq multiplicative property vector derivations}), we may define the desired inner product in $ \mathcal{R} $ as follows.

First, constants in $ \mathbb{F} $ act on the left as scalars ($ (a 1) F = a F $, for all $ F \in \mathcal{R} $). Now given $ \mathfrak{m}, \mathfrak{n} \in \mathcal{M} $, we define recursively on $ \mathfrak{m} $ the products
$$ (\mathfrak{m} x_i)(a \mathfrak{n}) = \sum_{j=1}^n \mathfrak{m} (\sigma_{i,j}(a) (x_j \mathfrak{n})) + \mathfrak{m}(\delta_i(a) \mathfrak{n}), $$
for all $ i = 1,2, \ldots, n $ and all $ a \in \mathbb{F} $, where $ \mathfrak{m} x_i $ and $ x_j \mathfrak{n} $ denote appending of monomials. Observe that this already defines, recursively on $ \mathfrak{m} $, the products of monomials as appending them, by choosing $ a = 1 $.

Finally, given general skew polynomials $ F = \sum_{\mathfrak{m} \in \mathcal{M}} F_\mathfrak{m} \mathfrak{m} $ and $ G = \sum_{\mathfrak{n} \in \mathcal{M}} G_\mathfrak{n} \mathfrak{n} $, where $ F_\mathfrak{m}, G_\mathfrak{m} \in \mathcal{R} $, for all $ \mathfrak{m} \in \mathcal{M} $, we define
$$ FG = \sum_{\mathfrak{m} \in \mathcal{M}} \sum_{\mathfrak{n} \in \mathcal{M}} F_\mathfrak{m} \left( \mathfrak{m} \left( G_\mathfrak{n} \mathfrak{n} \right) \right). $$

Note that this product is well-defined, since $ \deg(F) = d $ and $ \deg(G) = e $ imply that the coefficient of $ \mathfrak{m} $ in $ FG $ is zero whenever $ \deg(\mathfrak{m}) > d+e $, for all $ F, G \in \mathcal{R} $ and all $ \mathfrak{m} \in \mathcal{M} $. The properties of such an inner product stated in the theorem are all trivial, except for associativity, whose verification is left to the reader. 
\end{proof}

This motivates the following definitions:

\begin{definition}[\textbf{Matrix morphisms and vector derivations}] \label{def matrix mor and vector der}
We call every ring morphism $ \sigma : \mathbb{F} \longrightarrow \mathbb{F}^{n \times n} $ a matrix morphism (over $ \mathbb{F} $), and we say that a map $ \delta : \mathbb{F} \longrightarrow \mathbb{F}^n $ is a $ \sigma $-vector derivation (over $ \mathbb{F} $) if it is additive and satisfies 
$$ \delta(ab) = \sigma(a) \delta(b) + \delta(a)b, $$
for all $ a,b \in \mathbb{F} $.
\end{definition}

\begin{definition}[\textbf{Free multivariate skew polynomial rings}]
Given a matrix morphism $ \sigma : \mathbb{F} \longrightarrow \mathbb{F}^{n \times n} $ and a $ \sigma $-vector derivation $ \delta : \mathbb{F} \longrightarrow \mathbb{F}^n $, we define the free (multivariate) skew polynomial ring corresponding to $ \sigma $ and $ \delta $ as the unique ring $ \mathcal{R} = \mathbb{F}[\mathbf{x}; \sigma, \delta] $ with the inner product given by (\ref{eq def inner product}).
\end{definition}
 
Observe that the conventional free multivariate polynomial ring (called free algebra over $ \mathbb{F} $ when $ \mathbb{F} $ is commutative, see \cite[Sec. 0.11]{cohn} and \cite[Sec. 6.5]{cohn}) on the variables $ x_1, x_2, \ldots, $ $ x_n $ is obtained by choosing $ \sigma = {\rm Id} $ and $ \delta = 0 $, where we define $ {\rm Id}(a) = aI $, for all $ a \in \mathbb{F} $. Moreover, observe that this is the only case where constants and variables commute, which coincides with the only case where $ \mathbb{F}[\mathbf{x}; \sigma, \delta] $ is an algebra over $ \mathbb{F} $ when $ \mathbb{F} $ is commutative (here by \textit{algebra} we mean a ring $ \mathcal{R} $ that is a vector space over $ \mathbb{F} $ and whose inner product is $ \mathbb{F} $-bilinear, as in \cite{cohn}). Finally, observe also that $ \mathbb{F}[\mathbf{x}; \sigma, \delta] $ can still be characterized by a universal property similar to that of the free algebra. We only need to replace in the universal property the commutativity of constants and variables on free algebras by the rule (\ref{eq def inner product}). We leave the details to the reader. 

We conclude the section with some particular instances of matrix morphisms and vector derivations of interest:

\begin{example} \label{example diagonal case}
A matrix morphism $ \sigma : \mathbb{F} \longrightarrow \mathbb{F}^{n \times n} $ satisfies $ \sigma_{i,j}(a) = 0 $, for all $ a \in \mathbb{F} $ and all $ i \neq j $ if, and only if, there exist ring endomorphisms $ \sigma_i : \mathbb{F} \longrightarrow \mathbb{F} $, for $ i = 1,2, \ldots, n $, such that
$$ \sigma(a) = \left( \begin{array}{cccc}
\sigma_1(a) & 0 & \ldots & 0 \\
0 & \sigma_2(a) & \ldots & 0 \\
\vdots & \vdots & \ddots & \vdots \\
0 & 0 & \ldots & \sigma_n(a) \\
\end{array} \right), $$
for all $ a \in \mathbb{F} $. It is trivial to check that the family of $ \sigma $-vector derivations in this case are precisely those such that $ \delta_i $ is a $ \sigma_i $-derivation, for $ i = 1,2, \ldots, n $. An example is $ \mathbb{F} = k(t_1, t_2, \ldots, t_n) $, where $ k $ is a field, $ t_1, t_2, \ldots, t_n $ are algebraically independent variables, $ \sigma_i = {\rm Id} $ and $ \delta_i = \frac{\partial}{\partial t_i} $ is the conventional $ i $th partial derivative, for $ i = 1,2, \ldots, n $.
\end{example}

\begin{example} \label{example vector derivations}
Let $ \sigma : \mathbb{F} \longrightarrow \mathbb{F}^{n \times n} $ be a matrix morphism, and let $ \boldsymbol\beta \in \mathbb{F}^n $. The map $ \delta : \mathbb{F} \longrightarrow \mathbb{F}^n $ defined by
$$ \delta(a) = \sigma(a) \boldsymbol\beta - \boldsymbol\beta a, $$
for all $ a \in \mathbb{F} $, is a $ \sigma $-vector derivation. When $ n = 1 $, these vector derivations are called \textit{inner derivations} in the literature. 
\end{example}

\section{Evaluations of multivariate skew polynomials} \label{sec evaluations}

In this section, we show how to define \textit{evaluation maps} $ E_{\mathbf{a}} : \mathbb{F}[\mathbf{x}; \sigma, \delta] \longrightarrow \mathbb{F} $, for all $ \mathbf{a} \in \mathbb{F}^n $, that can be considered natural or standard. We will first require that these maps are left linear forms over $ \mathbb{F} $. We may then define the \textit{total evaluation map} as 
\begin{equation}
E : \mathbb{F}[\mathbf{x}; \sigma, \delta] \longrightarrow \mathbb{F}^{\mathbb{F}^n}: F \mapsto \left( E_{\mathbf{a}}(F) \right) _{\mathbf{a} \in \mathbb{F}^n},
\label{eq def total evaluation map}
\end{equation}
which is again left linear. By linearity, we have that
$$ E_{\mathbf{a}} \left( \sum_{\mathfrak{m} \in \mathcal{M}} F_\mathfrak{m} \mathfrak{m} \right) = \sum_{\mathfrak{m} \in \mathcal{M}} F_\mathfrak{m} N_\mathfrak{m}(\mathbf{a}), $$
for all $ \mathbf{a} \in \mathbb{F}^n $, all $ F_\mathfrak{m} \in \mathbb{F} $, and for functions 
$$ N_\mathfrak{m} : \mathbb{F}^n \longrightarrow \mathbb{F} : \mathbf{a} \longrightarrow E_\mathbf{a}(\mathfrak{m}) , $$
where $ \mathfrak{m} \in \mathcal{M} $. Therefore, giving a total evaluation map $ E $ is equivalent to giving the family of functions $ (N_\mathfrak{m})_{\mathfrak{m} \in \mathcal{M}} $, thus these will be called \textit{fundamental functions} of the evaluation $ E $. When $ n=1 $, the fundamental functions $ N_i = N_{x^i} $, for $ i = 0,1,2, \ldots $, coincide with those in \cite{lam, lam-leroy}.

As stated in Section \ref{sec intro}, a standard way of understanding evaluations of multivariate conventional polynomials is by giving a canonical ring isomorphism
$$ \mathbb{F} [x_1, x_2, \ldots, x_n] / \left( x_1-a_1, x_2-a_2, \ldots, x_n - a_n \right) \longrightarrow \mathbb{F}, $$
for all $ a_1, a_2, \ldots, a_n \in \mathbb{F} $, due to the ``Remainder Theorem''. The same idea is used in classical algebraic geometry to define evaluations as projections to a quotient ring given by a maximal ideal, which would be isomorphic to the so-called residue field.

To obtain such an isomorphism, we give a Euclidean-type division for skew polynomials of the type $ x_1-a_1, x_2-a_2, \ldots, x_n-a_n $:

\begin{lemma} \label{th euclidean division}
For any $ a_1, a_2, \ldots , a_n \in \mathbb{F} $ and any $ F \in \mathbb{F}[\mathbf{x}; \sigma, \delta] $, there exist unique $ G_1, G_2, \ldots, $ $ G_n \in \mathbb{F}[\mathbf{x}; \sigma, \delta] $ and $ b \in \mathbb{F} $ such that
\begin{equation}
F = \sum_{i = 1}^n G_i (x_i - a_i) + b.
\label{eq euclidean division}
\end{equation}
\end{lemma}
\begin{proof}
Existence is proven by a Euclidean division algorithm as usual. We next prove the uniqueness property. We only need to prove that if
\begin{equation}
\sum_{i = 1}^n G_i (x_i - a_i) + b = 0,
\label{eq euclidean division proof}
\end{equation}
then $ G_1 = G_2 = \ldots = G_n = b = 0 $. Assume the opposite. Without loss of generality, we may assume that $ G_n \neq 0 $ and $ \deg(G_n) \geq \deg(G_i) $, for all $ i $ with $ G_i \neq 0 $. 

Let $ \prec $ denote the graded lexicographic (from right to left) ordering in $ \mathcal{M} $ with $ x_1 \prec x_2 \prec \ldots \prec x_n $, and denote by $ {\rm LM}(G) \in \mathcal{M} $ the leading monomial of a skew polynomial $ G \in \mathbb{F}[\mathbf{x}; \sigma, \delta] $ with respect to $ \prec $. Then we see that the monomial $ {\rm LM}(G_n (x_n - a_n)) = {\rm LM}(G_n) x_n $ cannot be cancelled by any other monomial on the left-hand side of (\ref{eq euclidean division proof}). This is absurd and thus $ G_i = 0 $, for all $ i = 1,2, \ldots, n $. Hence $ b = 0 $ and we are done.
\end{proof}

\begin{remark}
Observe that the facts that the product in $ \mathbb{F}[\mathbf{x}; \sigma, \delta] $ consists in appending monomials and is additive on degrees are crucial in the proof of the previous lemma, since they allow us to state that $ {\rm LM}(G_n (x_n - a_n)) = {\rm LM}(G_n) x_n $ for the graded lexicographic ordering. These properties also ensure that the division algorithm does not run indefinitely. Note moreover that $ \mathbb{F} $ can be an arbitrary ring, since the leading coefficients of $ x_1-a_1, x_2-a_2, \ldots, x_n-a_n $ are all $ 1 $.
\end{remark}

\begin{remark} \label{remark why variables dont commute}
Observe that (being $ \mathbb{F} $ a division ring) we cannot guarantee that Lemma \ref{th euclidean division} (uniqueness of remainders) holds if we allow the variables to commute, unless we are dealing with multivariate conventional polynomials over fields.

Assume that $ n > 1 $ and add to the ring $ \mathcal{R} $ in Section \ref{sec matrix morphisms} the commutativity property on the variables: $ x_i x_j = x_j x_i $, for all $ i,j = 1,2, \ldots, n $. Observe that the rest of the properties of $ \mathcal{R} $ still imply the existence of the matrix morphism $ \sigma : \mathbb{F} \longrightarrow \mathbb{F}^{n \times n} $ and the $ \sigma $-vector derivation $ \delta : \mathbb{F} \longrightarrow \mathbb{F}^n $ by inspecting constants and variables.

Next take $ a_1, a_2, \ldots, a_n \in \mathbb{F} $. For fixed $ 1 \leq i < j \leq n $, we have that
$$ x_j (x_i - a_i) - x_i (x_j - a_j) = x_i a_j - x_j a_i = \sum_{k=1}^n \left( \sigma_{i,k}(a_j) - \sigma_{j,k}(a_i) \right) (x_k - a_k) $$
\begin{equation}
+ \left( \sum_{k=1}^n \left( \sigma_{i,k}(a_j) - \sigma_{j,k}(a_i) \right) a_k \right) + \delta_i(a_j) - \delta_j(a_i).
\label{eq term for noncommutative}
\end{equation}

Then the term (\ref{eq term for noncommutative}) equals $ 0 $ for all $ a_1, a_2, \ldots, a_n \in \mathbb{F} $ and all $ 0 \leq i < j \leq n $ if, and only if, $ \mathbb{F} $ is commutative, $ \sigma = {\rm Id} $ and $ \delta = 0 $. We leave the proof to the reader. 

In particular, unique remainder as in Lemma \ref{th euclidean division} can only be guaranteed in this case (variables commute and $ \mathbb{F} $ is a division ring) if $ \mathbb{F} $ is commutative, $ \sigma = {\rm Id} $ and $ \delta = 0 $. Therefore, evaluation by unique remainder division (thus ``plug-in'' evaluation, see Remark \ref{remark pluging non-commutative}) does not exist even for multivariate conventional polynomials with commutative variables over non-commutative division rings (as considered in \cite{amitsur}, for instance). 

However, one may usually define non-trivial relations between variables while preserving evaluation properties. See Section \ref{sec general skew}.
\end{remark}

\begin{remark} \label{remark why not iterated skew polynomials}
Since variables commute in many reasonable iterated skew polynomial rings, they do not satisfy the uniqueness of remainders as in Lemma \ref{th euclidean division}. Thus we do not consider iterated skew polynomials in this paper, in contrast with \cite{skewRM}. Take for instance any iterated skew polynomial ring $ \mathcal{S} = (\mathbb{F}[x_1;\sigma_1,\delta_1])[x_2;\sigma_2,\delta_2] $, where $ \delta_1 = \delta_2 = 0 $, $ \sigma_2(x_1) = x_1 $ and $ \sigma_1 \sigma_2 = \sigma_2 \sigma_1 $. Then $ x_2 x_1 = x_1 x_2 $. This can be easily extended to any number of variables. See also Remark \ref{remark iterated are also quotients}.
\end{remark}

We may now define a standard evaluation as follows, which extends the case $ n = 1 $ from \cite{lam, lam-leroy}:

\begin{definition}[\textbf{Standard evaluation}] \label{def standard evaluation}
For $ \mathbf{a} = (a_1, a_2, \ldots, a_n) \in \mathbb{F}^n $ and a skew polynomial $ F \in \mathbb{F} [\mathbf{x} ; \sigma,\delta] $, we define its $ (\sigma,\delta) $-evaluation, denoted by
\begin{equation}
F(\mathbf{a}) = E_{\mathbf{a}}^{\sigma, \delta}(F),
\label{eq def standard evaluation}
\end{equation}
as the unique element $ F(\mathbf{a}) \in \mathbb{F} $ such that 
$$ F - F(\mathbf{a}) \in \left( x_1-a_1, x_2 - a_2, \ldots, x_n - a_n \right). $$
We denote the corresponding total evaluation map by $ E^{\sigma, \delta} $, and we use the notations $ E_{\mathbf{a}} $ and $ E $ when there is no confusion about $ \sigma $ and $ \delta $.
\end{definition}

These evaluation maps are well-defined and left linear by Lemma \ref{th euclidean division}. To conclude, we give a recursive formula on the fundamental functions of the total evaluation map $ E^{\sigma, \delta} $, which is of computational interest. This result is an extension of the case $ n = 1 $ given in \cite[Lemma 2.4]{lam-leroy} and \cite[Eq. (2.3)]{lam-leroy}.

\begin{theorem} \label{th fundamental functions}
The fundamental functions $ N_\mathfrak{m}^{\sigma, \delta} = N_\mathfrak{m} : \mathbb{F}^n \longrightarrow \mathbb{F} $, for $ \mathfrak{m} \in \mathcal{M} $, of the $ (\sigma,\delta) $-evaluation $ E^{\sigma, \delta} $ in Definition \ref{def standard evaluation} are given recursively as follows: $ N_1(\mathbf{a}) = 1 $, and
\begin{equation}
\left( \begin{array}{c}
N_{x_1 \mathfrak{m}}(\mathbf{a}) \\
N_{x_2 \mathfrak{m}}(\mathbf{a}) \\
\vdots \\
N_{x_n \mathfrak{m}}(\mathbf{a}) \\
\end{array} \right) = \sigma(N_\mathfrak{m}(\mathbf{a})) \mathbf{a} + \delta (N_\mathfrak{m}(\mathbf{a})),
\label{eq fundamental functions standard ev}
\end{equation}
for all $ \mathfrak{m} \in \mathcal{M} $ and all $ \mathbf{a} \in \mathbb{F}^n $.
\end{theorem}
\begin{proof}
We will use the compact matrix/vector notation in (\ref{eq def inner product compact}), and we proceed recursively on $ \mathfrak{m} \in \mathcal{M} $, for fixed $ \mathbf{a} \in \mathbb{F}^n $. 

Obviously, $ N_1(\mathbf{a}) = 1 $. Assume now that it is true for a monomial $ \mathfrak{m} \in \mathcal{M} $. Therefore, there exist skew polynomials $ P_1, P_2, \ldots, P_n \in \left( x_1-a_1, x_2 - a_2, \ldots, x_n - a_n \right) $ such that, if we denote by $ \mathbf{P} $ the column vector whose $ i $th row is $ P_i $, for $ i = 1,2, \ldots, n $, then
$$ \mathbf{x} \mathfrak{m} = \mathbf{P} + \mathbf{x} N_\mathfrak{m}(\mathbf{a}) = \mathbf{P} + \sigma (N_\mathfrak{m}(\mathbf{a})) \mathbf{x} + \delta(N_\mathfrak{m}(\mathbf{a})) $$
$$ = \mathbf{P} + \sigma(N_\mathfrak{m}(\mathbf{a})) (\mathbf{x} - \mathbf{a}) + (\sigma(N_\mathfrak{m}(\mathbf{a})) \mathbf{a} + \delta(N_\mathfrak{m}(\mathbf{a}))), $$
and the result follows by Lemma \ref{th euclidean division}.
\end{proof}

\begin{remark} \label{remark pluging non-commutative}
Note that, when $ \sigma = {\rm Id} $ and $ \delta = 0 $, Theorem \ref{th fundamental functions} states that evaluation by unique remainder coincides with evaluation performed by ``plugging values'' in the variables but with reversed orders (see also Remark \ref{remark why variables dont commute}). For instance, the evaluation of $ x_1x_2 $ at $ (a_1,a_2) $ would be $ a_2a_1 $.
\end{remark}

We recall that in the case $ n=1 $ and $ \delta = 0 $, we have that $ N_i(a) = N_{x^i}(a) = \sigma^{i-1}(a) \cdots \sigma(a) a $, for $ i = 1,2, \ldots $, hence the notation $ N_\mathfrak{m} $ is a reminder of its similarity with the \textit{norm} function.

It has been recently shown \cite{lin-multivariateskew} that norms as in (\ref{eq fundamental functions standard ev}) allow to naturally state Hilbert's Theorem 90 for general Galois extensions of fields (as considered by Noether) using arbitrary generators and relations of the Galois group.

\section{Conjugacy and the product rule} \label{sec product rule}

From the previous section, we know that the $ (\sigma, \delta) $-evaluation $ E^{\sigma, \delta} $ is left linear. In this section, we will use the multiplicative properties of $ \sigma $, $ \delta $ and the fundamental functions of $ E^{\sigma, \delta} $ to show that it preserves products of skew polynomials in a certain way. This property will be used in the next section to define ideals of zeros and to characterize which of them are two-sided (Proposition \ref{prop two-sided ideals}). It will be especially important in Section \ref{sec lagrange interpolation} for constructing skew polynomials of restricted degree with a given set of zeros.

We need the concept of conjugacy, where the case $ n = 1 $ was given in \cite[Eq. (2.5)]{lam-leroy}.

\begin{definition}[\textbf{Conjugacy}]
Given $ \mathbf{a} \in \mathbb{F}^n $ and $ c \in \mathbb{F}^* $, we define the $ (\sigma,\delta) $-conjugate, or just conjugate if there is no confusion, of $ \mathbf{a} $ with respect to $ c $ as
\begin{equation}
\mathbf{a}^c = \sigma(c) \mathbf{a} c^{-1} + \delta(c) c^{-1} \in \mathbb{F}^n .
\label{eq def conjugate}
\end{equation}
\end{definition}

We have the following properties, which extend the case $ n = 1 $ given in \cite[Eq. (2.6)]{lam-leroy}.

\begin{lemma}
Given $ \mathbf{a}, \mathbf{b} \in \mathbb{F}^n $ and $ c,d \in \mathbb{F}^* $, the following properties hold:
\begin{enumerate}
\item
$ \mathbf{a}^1 = \mathbf{a} $ and $ (\mathbf{a}^c)^d = \mathbf{a}^{dc} $.
\item
The relation $ \mathbf{a} \thicksim \mathbf{b} $ if, and only if, there exist $ e \in \mathbb{F}^* $ with $ \mathbf{b} = \mathbf{a}^e $, is an equivalence relation on $ \mathbb{F}^n $.
\end{enumerate}
\end{lemma}

If $ n = 1 $, $ \sigma = {\rm Id} $ and $ \delta = 0 $, then the previous notion of conjugacy coincides with the usual one on the multiplicative monoid of $ \mathbb{F} $, which explains the terminology.

As noted in \cite{hilbert90}, Hilbert's Theorem 90 can be understood as any effective criterion for conjugacy, which in its classical form (cyclic Galois extensions) is given by the classical norm function. The same idea can be used to reinterpret Hilbert's Theorem 90 over any Galois extension of fields \cite{lin-multivariateskew}, where the norm function is replaced by the fundamental functions from the last section.

We may now establish and prove the product rule. The case $ n = 1 $ was first given in \cite[Th. 2.7]{lam-leroy}. We follow their proof using our matrix/vector notation.

\begin{theorem}[\textbf{Product rule}] \label{th product rule}
Given skew polynomials $ F, G \in \mathbb{F}[\mathbf{x}; \sigma, \delta] $ and $ \mathbf{a} \in \mathbb{F}^n $, if $ G(\mathbf{a}) = 0 $, then $ (FG)(\mathbf{a}) = 0 $, and if $ c = G(\mathbf{a}) \neq 0 $, then
\begin{equation}
(FG)(\mathbf{a}) = F(\mathbf{a}^c) G(\mathbf{a}).
\label{eq product rule}
\end{equation}
\end{theorem}
\begin{proof}
It is obvious from Lemma \ref{th euclidean division} and Definition \ref{def standard evaluation} that, if $ G(\mathbf{a}) = 0 $, then $ (FG)(\mathbf{a}) = 0 $. Now assume that $ c = G(\mathbf{a}) \neq 0 $. First observe that
$$ \left( \mathbf{x} - \mathbf{a}^c \right) c = \sigma(c) (\mathbf{x} - \mathbf{a}). $$

Second, by Definition \ref{def standard evaluation} there exist skew polynomials $ P_i, Q_i \in \mathbb{F}[\mathbf{x}; \sigma, \delta] $, for $ i = 1,2, \ldots, n $, such that
$$ F = \mathbf{P}^T \left( \mathbf{x} - \mathbf{a}^c \right) + F \left( \mathbf{a}^c \right), \textrm{ and } G = \mathbf{Q}^T (\mathbf{x} - \mathbf{a}) + G(\mathbf{a}), $$
where $ \mathbf{P} $ and $ \mathbf{Q} $ denote the column vectors whose $ i $th rows are $ P_i $ and $ Q_i $, respectively, for $ i = 1,2, \ldots, n $. Combining these facts, we obtain that
$$ FG = F \mathbf{Q}^T (\mathbf{x} - \mathbf{a}) + F G(\mathbf{a}) $$
$$ = \left( F \mathbf{Q}^T + \mathbf{P}^T \sigma(c) \right) (\mathbf{x} - \mathbf{a}) + F(\mathbf{a}^c)G(\mathbf{a}), $$
and we are done.
\end{proof}

This theorem can be stated when $ \mathbb{F} $ is an arbitrary ring by considering only the cases where $ c=0 $ or $ c $ is a unit. The fact that only one of these two cases happen when $ \mathbb{F} $ is a division ring will be crucial in Proposition \ref{prop two-sided ideals} and from Section \ref{sec lagrange interpolation} onwards.

Note that $ \mathbf{a}^c \neq \mathbf{a} $ in general when $ \mathbb{F} $ is non-commutative even if $ \sigma = {\rm Id} $ and $ \delta = 0 $. Thus the product rule is still of value for conventional polynomials over division rings. In particular, it may be that $ F(\mathbf{a}) = 0 $ and $ (FG)(\mathbf{a}) \neq 0 $ even for conventional polynomials, when $ \mathbb{F} $ is non-commutative.

\section{Zeros of multivariate skew polynomials} \label{sec zeros}

In this section, we will define and give the basic properties of sets of zeros of multivariate skew polynomials and, conversely, sets of skew polynomials that vanish at a certain set of affine points, which will be crucial in Section \ref{sec lagrange interpolation} for Lagrange interpolation. Conceptually, they will also be important in Section \ref{sec general skew} to define general skew polynomial rings with relations on the variables (\textit{nonfree}) and where evaluation still works in a natural way.

Observe that at this point our theory loses most of its analogies with the univariate case \cite{lam, algebraic-conjugacy, lam-leroy}, since $ \mathbb{F}[\mathbf{x}; \sigma, \delta] $ is not a principal ideal domain if $ n > 1 $, hence the use of \textit{minimal skew polynomials} as in \cite{lam, algebraic-conjugacy} is not possible. On the other hand, we gain analogy with respect to classical algebraic geometry:

\begin{definition} [\textbf{Zeros of skew polynomials}]
Given a set $ A \subseteq \mathbb{F}[\mathbf{x}; \sigma, \delta] $, we define its zero set as
$$ Z(A) = \{ \mathbf{a} \in \mathbb{F}^n \mid F(\mathbf{a}) = 0, \forall F \in A \}. $$
And given a set $ \Omega \subseteq \mathbb{F}^n $, we define its associated ideal as
$$ I(\Omega) = \{ F \in \mathbb{F}[\mathbf{x}; \sigma, \delta] \mid F(\mathbf{a}) = 0, \forall \mathbf{a} \in \Omega \}. $$
\end{definition}

Observe that the ideal associated to a subset of $ \mathbb{F}^n $ is indeed a left ideal:

\begin{proposition}
For any $ \Omega \subseteq \mathbb{F}^n $, it holds that $ I(\Omega) \subseteq \mathbb{F}[\mathbf{x}; \sigma, \delta] $ is a left ideal.
\end{proposition}
\begin{proof}
It follows directly from the product rule (Theorem \ref{th product rule}). Alternatively, it can be proven by noting that $ I(\Omega) = \bigcap_{\mathbf{a} \in \Omega} (x_1 - a_1, x_2 - a_2, \ldots, x_n - a_n) $.
\end{proof}

We next list some basic properties of zero sets and ideals of zeros that follow from the definitions, in the same way as in classical algebraic geometry.

\begin{proposition} \label{prop properties of zeros}
Let $ \Omega, \Omega_1, \Omega_2 \subseteq \mathbb{F}^n $ and $ A, A_1, A_2 \subseteq \mathbb{F}[\mathbf{x}; \sigma, \delta] $ be arbitrary sets. The following properties hold:
\begin{enumerate}
\item
$ I(\{ \mathbf{a} \}) = \left( x_1-a_1, x_2 - a_2, \ldots, x_n - a_n \right) $ and $ Z(x_1-a_1, x_2 - a_2, \ldots, x_n - a_n) = \{ \mathbf{a} \} $, for all $ \mathbf{a} = (a_1, a_2, \ldots, a_n) \in \mathbb{F}^n $.
\item
$ I(\varnothing) = \left( 1 \right) $ and $ Z( 1 ) = \varnothing $.
\item
$ I(\mathbb{F}^n) \subseteq I(\Omega) $ and $ Z(\{ 0 \}) = \mathbb{F}^n $. That is, $ I(\mathbb{F}^n) $ is the minimal ideal of zeros.
\item
If $ \Omega_1 \subseteq \Omega_2 $, then $ I(\Omega_2) \subseteq I(\Omega_1) $.
\item
If $ A_1 \subseteq A_2 $, then $ Z(A_2) \subseteq Z(A_1) $.
\item
$ I(\Omega_1 \cup \Omega_2) = I(\Omega_1) \cap I(\Omega_2) $.
\item
$ Z(A) = Z(\left( A \right)) $ and $ Z(A_1 \cup A_2) = Z(\left( A_1 \right) + \left( A_2 \right)) = Z(A_1) \cap Z(A_2) $.
\item
$ \Omega \subseteq Z(I(\Omega)) $, and equality holds if, and only if, $ \Omega = Z(B) $ for some $ B \subseteq \mathbb{F}[\mathbf{x}; \sigma, \delta] $.
\item
$ A \subseteq \left( A \right) \subseteq I(Z(A)) $, and equality holds if, and only if, $ A = I(\Psi) $ for some $ \Psi \subseteq \mathbb{F}^n $.
\end{enumerate}
\end{proposition}

Item 8 in the previous proposition motivates the definition of \textit{P-closed sets}, where the case $ n = 1 $ was given in \cite{lam, lam-leroy}:

\begin{definition}[\textbf{P-closures}]
Given a subset $ \Omega \subseteq \mathbb{F}^n $, we define its P-closure as 
$$ \overline{\Omega} = Z(I(\Omega)), $$
and we say that $ \Omega $ is P-closed if $ \overline{\Omega} = \Omega $.
\end{definition}

By Proposition \ref{prop properties of zeros}, Item 8, P-closed sets correspond to sets of zeros of sets of skew polynomials, and we have the following:

\begin{lemma}
Given a subset $ \Omega \subseteq \mathbb{F}^n $, it holds that $ \overline{\Omega} $ is the smallest P-closed subset of $ \mathbb{F}^n $ containing $ \Omega $. 
\end{lemma}

\section{General and minimal skew polynomial rings} \label{sec general skew}

In this section, we define \textit{general skew polynomial rings} as those with a set of relations on the variables and where evaluation is still as in Definition \ref{def standard evaluation}. In particular, by considering a maximum set of such relations, we may define \textit{minimal skew polynomial rings}.

Note that the whole space $ \mathbb{F}^n $ is P-closed, and Item 3 in Proposition \ref{prop properties of zeros} says that, for evaluation purposes, we may just consider the quotient left module
$$ \mathbb{F}[\mathbf{x}; \sigma, \delta] / I(\mathbb{F}^n), $$
which is a ring if $ I(\mathbb{F}^n) $ is a two-sided ideal, and in such a case we obtain the above mentioned minimal skew polynomial ring where the $ (\sigma, \delta) $-standard evaluation is still defined.

In the following proposition, we characterize when an ideal of zeros is two-sided, which includes in particular the ideal $ I(\mathbb{F}^n) $:

\begin{proposition} \label{prop two-sided ideals}
Given a subset $ \Omega \subseteq \mathbb{F}^n $, the following are equivalent:
\begin{enumerate}
\item
$ I(\Omega) $ is a two-sided ideal.
\item
If $ F \in I(\Omega) $ and $ c \in \mathbb{F} $, then $ Fc \in I(\Omega) $.
\item
If $ \mathbf{a} \in \overline{\Omega} $, then $ \mathbf{a}^c \in \overline{\Omega} $, for all $ c \in \mathbb{F}^* $.
\item
If $ \mathbf{a} \in \Omega $, then $ \mathbf{a}^c \in \overline{\Omega} $, for all $ c \in \mathbb{F}^* $.
\end{enumerate}
In particular, $ I(\mathbb{F}^n) $ is a two-sided ideal.
\end{proposition}
\begin{proof}
We prove the following implications:

$ 1) \Longrightarrow 2) $: Trivial.

$ 2) \Longrightarrow 3) $: Let $ \mathbf{a} \in \overline{\Omega} $, $ F \in I(\Omega) $ and $ c \in \mathbb{F}^* $. First, it holds that $ I(\Omega) = I(\overline{\Omega}) $ by Items 8 and 9 in Proposition \ref{prop properties of zeros}, and $ Fc \in I(\Omega) $ by hypothesis. Thus
$$ 0 = (Fc)(\mathbf{a}) = F(\mathbf{a}^c) c $$
by the product rule (Theorem \ref{th product rule}). Hence $ \mathbf{a}^c \in Z(I(\Omega)) = \overline{\Omega} $.

$ 3) \Longrightarrow 4) $: Trivial from $ \Omega \subseteq \overline{\Omega} $.

$ 4) \Longrightarrow 1) $: Let $ F \in I(\Omega) $ and $ G \in \mathbb{F}[\mathbf{x}; \sigma, \delta] $, fix $ \mathbf{a} \in \Omega $ and define $ c = G(\mathbf{a}) $. If $ c = 0 $, then $ (FG)(\mathbf{a}) = 0 $ by the product rule. If $ c \neq 0 $, by hypothesis and the product rule, we have that
$$ (FG)(\mathbf{a}) = F(\mathbf{a}^c) G(\mathbf{a}) = 0, $$
since $ \mathbf{a}^c \in \overline{\Omega} $ and $ F \in I(\Omega) = I(\overline{\Omega}) $. Hence $ (FG)(\mathbf{a}) = 0 $ for any $ \mathbf{a} \in \Omega $ and thus $ FG \in I(\Omega) $, and we are done.
\end{proof}

Observe that, to prove $ 4) \Longrightarrow 1) $, we use that $ \mathbb{F} $ is a division ring, since we use that every $ c \in \mathbb{F} \setminus \{ 0 \} $ is invertible.

We may now define (nonfree) general skew polynomial rings and, in particular, a minimal one.

\begin{definition} [\textbf{Skew polynomial rings}] \label{def skew polynomial rings}
For any two-sided ideal $ I \subseteq I(\mathbb{F}^n) $, we say that the quotient ring
$$ \mathbb{F}[\mathbf{x}; \sigma, \delta] / I $$
is a skew polynomial ring with matrix morphism $ \sigma $ and vector derivation $ \delta $. The minimal skew polynomial ring with matrix morphism $ \sigma $ and vector derivation $ \delta $ is defined as that obtained when $ I = I(\mathbb{F}^n) $.
\end{definition}

This is exactly what happens with multivariate conventional polynomial rings (the case $ \sigma = {\rm Id} $ and $ \delta = 0 $) over fields. One may consider the free multivariate polynomial ring and define the conventional evaluation on it, either by plugging values in the variables or equivalently by unique remainder division. When $ \mathbb{F} $ is a field, $ I(\mathbb{F}^n) $ contains the two-sided ideal $ J $ generated by $ x_ix_j - x_jx_i $, for $ 1 \leq i < j \leq n $.  If $ \mathbb{F} $ is infinite, then $ J = I(\mathbb{F}^n) $, whereas 
\begin{equation}
I(\mathbb{F}^n) = J + ( x_1^q - x_1, x_2^q - x_2, \ldots, x_n^q - x_n )
\label{eq min conventional skew pol ring for finite}
\end{equation}
if $ \mathbb{F} $ is finite and has $ q $ elements. The results in the following sections will be proven for the free multivariate skew polynomial ring. By projecting onto the quotient, they can also be stated for any multivariate skew polynomial ring.

We conclude this section by showing that skew polynomial rings can form iterated sequences of rings by adding variables, even though we do not consider iterated skew polynomial rings in the standard way, as shown in Remark \ref{remark why not iterated skew polynomials}. The proof of the following result is straightforward.

\begin{proposition}
Let $ 0 < r < n $ be a positive integer, let $ \tau : \mathbb{F} \longrightarrow \mathbb{F}^{r \times r} $ and $ \nu : \mathbb{F} \longrightarrow \mathbb{F}^{(n-r) \times (n-r)} $ be matrix morphisms and let $ \delta_\tau : \mathbb{F} \longrightarrow \mathbb{F}^r $ and $ \delta_\nu : \mathbb{F} \longrightarrow \mathbb{F}^{(n-r)} $ be a $ \tau $-vector derivation and a $ \nu $-vector derivation, respectively. Define $ \sigma : \mathbb{F} \longrightarrow \mathbb{F}^{n \times n} $ and $ \delta_\sigma : \mathbb{F} \longrightarrow \mathbb{F}^n $ by
$$ \sigma(a) = \left( \begin{array}{cc}
\tau(a) & 0 \\
0 & \nu(a)
\end{array} \right) \quad \textrm{and} \quad \delta_\sigma(a) = \left( \begin{array}{c}
\delta_\tau(a) \\
\delta_\nu(a)
\end{array} \right), $$
for all $ a \in \mathbb{F} $. Then $ \sigma $ is a matrix morphism and $ \delta_\sigma $ is a $ \sigma $-vector derivation. Consider now the natural inclusion map
$$ \rho : \mathbb{F}[x_1, x_2, \ldots, x_r ; \tau, \delta_\tau] \longrightarrow \mathbb{F}[x_1, x_2, \ldots, x_n ; \sigma, \delta_\sigma]. $$
The following properties hold:
\begin{enumerate}
\item
$ \rho $ is a one to one ring morphism.
\item
For all $ F \in \mathbb{F}[x_1, x_2, \ldots, x_r ; \tau, \delta_\tau] $, all $ \mathbf{a}_\tau \in \mathbb{F}^r $ and all $ \mathbf{a}_\nu \in \mathbb{F}^{n-r} $, it holds that
$$ E_{\mathbf{a}_\tau}^{\tau, \delta_\tau}(F) = E_{\mathbf{a}_\sigma}^{\sigma, \delta_\sigma}(\rho(F)), $$
where $ \mathbf{a}_\sigma = (\mathbf{a}_\tau, \mathbf{a}_\nu) \in \mathbb{F}^n $.
\item
For any two-sided ideal $ J \subseteq I(\mathbb{F}^n) $, it holds that $ \rho^{-1}(J) \subseteq I(\mathbb{F}^r) $ is a two-sided ideal and $ \rho $ can be restricted to a one to one ring morphism
$$ \rho : \mathbb{F}[x_1, x_2, \ldots, x_r ; \tau, \delta_\tau] / \rho^{-1}(J) \longrightarrow \mathbb{F}[x_1, x_2, \ldots, x_n ; \sigma, \delta_\sigma] / J. $$
This holds in particular choosing $ J = I(\mathbb{F}^n) $, which implies that $ \rho^{-1}(J) = I(\mathbb{F}^r) $.
\end{enumerate}
\end{proposition}

In particular, if $ \sigma $ and $ \delta $ are given as in Example \ref{example diagonal case}, then $ \mathbb{F}[\mathbf{x};\sigma, \delta] $ contains a sequence of $ n $ nested skew polynomial rings, where the first one is the univariate skew polynomial ring $ \mathbb{F}[x_1; \sigma_1, \delta_1] $.

\begin{remark} \label{remark iterated are also quotients}
Iterated skew polynomial rings such that $ \delta_i(\mathbb{F}) \subseteq \mathbb{F} + \mathbb{F}x_1 + \cdots + \mathbb{F}x_{i-1} $, for $ i = 1,2, \ldots, n $, are also quotients of free multivariate skew polynomial rings since they satisfy the rules (\ref{eq def inner product}), setting $ \sigma_{i,i}(a) = \sigma_i(a) $ and $ \sigma_{i,j}(a) $ as the coefficient of $ x_j $ in $ \delta_i(a) $ for $ j<i $ (note that necessarily $ \sigma_i(\mathbb{F}) \subseteq \mathbb{F} $, for $ i = 1,2, \ldots, n $). Use for instance the universal property, as explained in Section \ref{sec matrix morphisms}. Examples include those in Remark \ref{remark why not iterated skew polynomials} or \cite[Ex. 2.3]{skewRM}, and important rings such as Weyl algebras \cite{galligo-diff} or solvable iterated skew polynomial rings \cite{kandri, zhangPBW}. In particular, when evaluation can be given for such iterated skew polynomials by unique remainder as in \cite{skewRM}, it must coincide with our notion of evaluation (Definition \ref{def standard evaluation}). What happens is that these iterated skew polynomial rings are in general quotients by a two-sided ideal $ J $ that satisfies that $ J \setminus I(\mathbb{F}^n) \neq \varnothing $ and $ I(\mathbb{F}^n) \setminus J \neq \varnothing $.
\end{remark}

\section{P-generators, P-independence and P-bases} \label{sec P-bases}

The main feature of P-closed sets is that they can be ``generated'' by certain subsets, called \textit{P-bases}, that control the possible values given by a function defined by a skew polynomial on such sets, as we will show in the next section. P-bases are given by a \textit{P-independence} notion and naturally form a matroid (Proposition \ref{prop matroid}). P-independence was defined for the case $ n=1 $ in \cite{lam, algebraic-conjugacy}. We start with the main definitions:

\begin{definition}[\textbf{P-generators}] \label{def P-generators}
Given a P-closed set $ \Omega \subseteq \mathbb{F}^n $, we say that $ \mathcal{G} \subseteq \Omega $ generates $ \Omega $ if $ \overline{\mathcal{G}} = \Omega $, and it is then called a set of P-generators for $ \Omega $. We say that $ \Omega $ is finitely generated if it has a finite set of P-generators.
\end{definition}

\begin{definition}[\textbf{P-independence}]
We say that $ \mathbf{a} \in \mathbb{F}^n $ is P-independent from $ \Omega \subseteq \mathbb{F}^n $ if it does not belong to $ \overline{\Omega} $. 

A set $ \Omega \subseteq \mathbb{F}^n $ is called P-independent if every $ \mathbf{a} \in \Omega $ is P-independent from $ \Omega \setminus \{ \mathbf{a} \} $. P-dependent means not P-independent.
\end{definition}

\begin{definition}[\textbf{P-bases}]
Given a P-closed set $ \Omega \subseteq \mathbb{F}^n $, we say that a subset $ \mathcal{B} \subseteq \Omega $ is a P-basis of $ \Omega $ if it is P-independent and $ \overline{\mathcal{B}} = \Omega $.
\end{definition}

The following is the main result of this section, where Item 3 will be crucial in order to perform Lagrange interpolation recursively.

\begin{proposition} \label{prop characterizations P-bases}
Given sets $ \mathcal{B} \subseteq \Omega \subseteq \mathbb{F}^n $, where $ \Omega = \overline{\mathcal{B}} $, the following are equivalent:
\begin{enumerate}
\item
$ \mathcal{B} $ is a P-basis of $ \Omega $.
\item
If $ \mathcal{G} \subseteq \mathcal{B} $ and $ \overline{\mathcal{G}} = \Omega $, then $ \mathcal{G} = \mathcal{B} $. That is, $ \mathcal{B} $ is a minimal set of P-generators of $ \Omega $.
\item
(If $ \mathcal{B} $ is finite) For any ordering $ \mathbf{b}_1, \mathbf{b}_2, \ldots, \mathbf{b}_M $ of the elements in $ \mathcal{B} $ and for $ i = 0,1,2, \ldots, M-1 $, it holds that $ \mathbf{b}_{i+1} $ is P-independent from $ \mathcal{B}_i = \{ \mathbf{b}_1, \mathbf{b}_2, \ldots, $ $ \mathbf{b}_i \} $, where $ \mathcal{B}_0 = \varnothing $.
\end{enumerate}
\end{proposition}
\begin{proof}
We prove each implication separately:

$ 1) \Longrightarrow 2) $: Assume that there exists $ \mathcal{G} \subsetneqq \mathcal{B} $ with $ \overline{\mathcal{G}} = \Omega $ and let $ \mathbf{a} \in \mathcal{B} \setminus \mathcal{G} $. Then 
$$ \mathbf{a} \in \Omega = \overline{\mathcal{G}} = \overline{\mathcal{B} \setminus \{ \mathbf{a} \}}, $$
hence Item 1 does not hold. 

$ 2) \Longrightarrow 1) $: Assume that $ \mathcal{B} $ is not P-independent and take $ \mathbf{a} \in \mathcal{B} $ with $ \mathbf{a} \in \overline{\mathcal{B} \setminus \{ \mathbf{a} \}} $. Define $ \mathcal{G} = \mathcal{B} \setminus \{ \mathbf{a} \} \subsetneqq \mathcal{B} $. It holds that
$$ \mathcal{B} = \{ \mathbf{a} \} \cup (\mathcal{B} \setminus \{ \mathbf{a} \}) \subseteq \overline{\mathcal{G}}, $$
hence $ \overline{\mathcal{G}} = \Omega $ and Item 2 does not hold.

$ 1) \Longrightarrow 3) $: Assume that $ \mathbf{b}_{i+1} $ is P-dependent from $ \mathcal{B}_i $ for a given $ i $ and a given ordering of $ \mathcal{B} $. Then 
$$ \mathbf{b}_{i+1} \in \overline{\mathcal{B}_i} \subseteq \overline{\mathcal{B} \setminus \{ \mathbf{b}_{i+1} \}}, $$
hence Item 1 does not hold.

$ 3) \Longrightarrow 1) $: Assume that $ \mathbf{a} $ is P-dependent from $ \mathcal{B} \setminus \{ \mathbf{a} \} $ and order the $ M $ elements in $ \mathcal{B} $ in such a way that $ \mathbf{b}_M = \mathbf{a} $. Then $ \mathbf{b}_M $ is P-dependent from $ \mathcal{B}_{M-1} $ and Item 3 does not hold.
\end{proof}

We have the following important immediate consequence of Item 2 in the previous proposition:

\begin{corollary} \label{corollary finite basis}
If a P-closed set is finitely generated, then it admits a finite P-basis.
\end{corollary}

Finally, we observe that the family of P-independent sets forms a matroid \cite[Sec. 1.1]{oxley} whose bases \cite[Sec. 1.2]{oxley} are precisely the family of P-bases. The proof requires results from the following sections, but we will state the observation in this section for clarity.

\begin{proposition} \label{prop matroid}
For every finitely generated P-closed set $ \Omega \subseteq \mathbb{F}^n $, the pair $ (\mathcal{P}(\Omega), \mathcal{I}_\Omega) $ forms a matroid, where $ \mathcal{P}(\Omega) $ is the collection of all subsets of $ \Omega $, and $ \mathcal{I}_\Omega $ is the collection of P-independent subsets of $ \Omega $. Furthermore, the bases of the matroid $ (\mathcal{P}(\Omega), \mathcal{I}_\Omega) $ are precisely the P-bases of $ \Omega $.
\end{proposition}
\begin{proof}
First, it is trivial to see that $ \varnothing \in \mathcal{I}_\Omega $ and, if $ \mathcal{A}^\prime \subseteq \mathcal{A} $ and $ \mathcal{A} \in \mathcal{I}_\Omega $, then $ \mathcal{A}^\prime \in \mathcal{I}_\Omega $. The augmentation property of matroids is the first statement in Lemma \ref{lemma adding one to P-independent}, proven in Section \ref{sec image and ker}. Finally, the fact that bases (that is, maximal independent sets) and P-bases coincide is the second statement in Lemma \ref{lemma adding one to P-independent}.
\end{proof}

\section{Skew polynomial functions and Lagrange interpolation} \label{sec lagrange interpolation}

In this section, we give the main result of this paper. We show, in a Lagrange-type interpolation theorem, what values a function given by a skew polynomial can take when evaluated on a finitely generated P-closed set (Theorem \ref{th lagrange interpolation}). This result will be crucial in the following sections to describe the image and kernel of the evaluation map defined below (Theorem \ref{th describing evaluation as vector space}), to prove later that P-independent sets form a matroid (Lemma \ref{lemma adding one to P-independent}) and that the rank of a P-closed set is the rank of the corresponding skew Vandermonde matrix (Proposition \ref{prop rank vandermonde}). On the way, we derive other important results on P-closed sets.

Observe first that the total $ (\sigma,\delta) $-evaluation gives a left linear map
\begin{equation}
E^{\sigma,\delta}_{\Omega} : \mathbb{F}[\mathbf{x}; \sigma, \delta] \longrightarrow \mathbb{F}^{\Omega},
\label{eq restricted total evaluation}
\end{equation}
when restricted to evaluating over a subset $ \Omega \subseteq \mathbb{F}^n $ or, in other words, by composing $ E^{\sigma,\delta}_{\Omega} = \pi_\Omega \circ E^{\sigma,\delta} $, where $ \pi_\Omega : \mathbb{F}^{\mathbb{F}^n} \longrightarrow \mathbb{F}^\Omega $ is the canonical projection map.

Hence $ E^{\sigma, \delta}_{\Omega} $ gives a correspondence between multivariate skew polynomials $ F \in \mathbb{F}[\mathbf{x}; \sigma, \delta] $ and some particular functions $ f = E^{\sigma, \delta}_{\Omega}(F) : \Omega \longrightarrow \mathbb{F} $. Such functions will be called \textit{multivariate skew polynomial functions} over $ \Omega $.

Formally, the objective of this section and the next one is to describe the kernel and image of the map $ E^{\sigma, \delta}_{\Omega} $ when $ \Omega $ is P-closed and finitely generated. We start with the following lemma, which is a key tool in Lagrange interpolation:

\begin{lemma} \label{lemma previous to lagrange}
Let $ \mathcal{B} \subseteq \mathbb{F}^n $ be a finite P-independent set and let $ \mathbf{b} \notin \overline{\mathcal{B}} $. There exists $ F \in I(\mathcal{B}) \setminus I(\mathcal{B} \cup \{ \mathbf{b} \}) $ such that $ \deg(F) \leq \# \mathcal{B} $.
\end{lemma}
\begin{proof}
First we prove that $ I(\mathcal{B}) \setminus I(\mathcal{B} \cup \{ \mathbf{b} \}) \neq \varnothing $. Assume the opposite. Then
$$ \overline{\mathcal{B}} = Z(I(\mathcal{B})) = Z(I(\mathcal{B} \cup \{ \mathbf{b} \})), $$
and $ \mathcal{B} \cup \{ \mathbf{b} \} \subseteq Z(I(\mathcal{B} \cup \{ \mathbf{b} \})) $ by Item 8 in Proposition \ref{prop properties of zeros}. Thus $ \mathbf{b} \in \overline{\mathcal{B}} $, which is a contradiction.

Now let $ \mathcal{B} = \{ \mathbf{b}_1, \mathbf{b}_2, \ldots, \mathbf{b}_M \} $ with $ M = \# \mathcal{B} $, let $ \prec $ be any ordering of $ \mathcal{M} $ preserving degrees, and take $ F \in I(\mathcal{B}) \setminus I(\mathcal{B} \cup \{ \mathbf{b} \}) $ such that $ {\rm LM}(F) $ is minimum possible with respect to $ \prec $. Assume that $ \deg(F) \geq M+1 $, which implies that $ \deg({\rm LM}(F)) \geq M+1 $ by the choice of the ordering $ \prec $.

Let $ {\rm LM}(F) = \mathfrak{m} x_{i_1} x_{i_2} \cdots x_{i_{M+1}} $, for some $ \mathfrak{m} \in \mathcal{M} $. By the product rule (Theorem \ref{th product rule}), we may choose elements $ a_1, a_2, \ldots, a_{M+1} \in \mathbb{F} $ such that
$$ G = \mathfrak{m} (x_{i_1} - a_1) (x_{i_2} - a_2) \cdots (x_{i_{M+1}} - a_{M+1}) $$
satisfies that $ G(\mathbf{b}_i) = 0 $, for $ i = 1,2, \ldots, M+1 $, denoting $ \mathbf{b}_{M+1} = \mathbf{b} $. In particular, there exists $ a \in \mathbb{F} $ such that $ H = F - aG $ satisfies $ {\rm LM}(H) \prec {\rm LM}(F) $, since $ {\rm LM}(F) = {\rm LM}(G) $. Now, by the definition of $ G $, it holds that
$$ H = F - aG \in I(\mathcal{B}) \setminus I(\mathcal{B} \cup \{ \mathbf{b} \}), $$
which is absurd by the minimality of $ {\rm LM}(F) $. Therefore $ \deg(F) \leq M $ and we are done.
\end{proof}

\begin{remark}
Note that, to construct $ G $ in the previous proof, we are implicitly using that $ \mathbb{F} $ is a division ring and we are implicitly applying Theorem \ref{th product rule} in its full form. Note however that constructing $ G $ is only needed to ensure the bound $ \deg(F) \leq \# \mathcal{B} $, but the existence of $ F $ still holds if $ \mathbb{F} $ is an arbitrary ring. The bound on $ \deg(F) $ will only be used to define skew Vandermonde matrices in Section \ref{sec skew vandermonde}. We do not investigate the full validity of Lemma \ref{lemma previous to lagrange} when $ \mathbb{F} $ is an arbitrary ring.
\end{remark}

The main result of this section is a Lagrange-type interpolation theorem in $ \mathbb{F}[\mathbf{x}; \sigma, \delta] $, whose proof is given by an iterative Newton-type algorithm thanks to Item 3 in Proposition \ref{prop characterizations P-bases}. This result extends the case $ n = 1 $ given in \cite[Th. 8]{lam} (see also the beginning of \cite[Sec. 5]{lam-leroy}). Newton-type iterative algorithms have been given in \cite{zhang} for univariate skew polynomials, and in \cite{skew-interpolation} for their free left modules.

\begin{theorem}[\textbf{Lagrange interpolation}] \label{th lagrange interpolation}
Let $ \Omega \subseteq \mathbb{F}^n $ be a finitely generated P-closed set with finite P-basis $ \mathcal{B} = \{ \mathbf{b}_1, \mathbf{b}_2, \ldots, \mathbf{b}_M \} $. The following hold:
\begin{enumerate}
\item
If $ E^{\sigma, \delta}_{\mathcal{B}}(F) = E^{\sigma, \delta}_{\mathcal{B}}(G) $, then $ E^{\sigma, \delta}_{\Omega}(F) = E^{\sigma, \delta}_{\Omega}(G) $, for all $ F,G \in \mathbb{F}[\mathbf{x}; \sigma, \delta] $. That is, the values of a skew polynomial function $ f : \Omega \longrightarrow \mathbb{F} $ are uniquely given by $ f(\mathbf{b}_1), f(\mathbf{b}_2), \ldots, f(\mathbf{b}_M) $.
\item
For every $ a_1, a_2, \ldots, a_M \in \mathbb{F} $, there exists $ F \in \mathbb{F}[\mathbf{x}; \sigma, \delta] $ such that $ \deg(F) < M $ and $ F(\mathbf{b}_i) = a_i $, for $ i = 1,2, \ldots, M $.
\end{enumerate}
\end{theorem}
\begin{proof}
We prove each item separately.
\begin{enumerate}
\item
We just need to prove that $ E^{\sigma, \delta}_{\mathcal{B}}(F) = 0 $ implies $ E^{\sigma, \delta}_{\Omega}(F) = 0 $. By definition, $ \mathcal{B} \subseteq Z(F) $, and by Proposition \ref{prop properties of zeros}, it holds that $ I(Z(F)) \subseteq I(\mathcal{B}) $ and 
$$ \Omega = \overline{\mathcal{B}} = Z(I(\mathcal{B})) \subseteq Z(I(Z(F))) = Z(F), $$
and the result follows.
\item
Let $ \mathcal{B} = \{ \mathbf{b}_1, \mathbf{b}_2, \ldots, \mathbf{b}_M \} $ as in Proposition \ref{prop characterizations P-bases}, Item 3. We prove the result iteratively for each of the P-independent sets $ \mathcal{B}_i $, $ i = 1,2, \ldots, M $, as in Newton's algorithm.

We start by defining the skew polynomial $ F_1 = a_1 $, which obviously satisfies $ F_1(\mathbf{b}_1) = a_1 $ and $ \deg(F_1) < 1 $. Now assume that $ M > 1 $, $ 1 \leq i \leq M-1 $ and there exists a skew polynomial $ F_i $ such that $ F_i(\mathbf{b}_j) = a_j $, for $ j = 1,2, \ldots, i $, and $ \deg(F_i) < i $. By Lemma \ref{lemma previous to lagrange}, there exists
$$ G \in I(\{ \mathbf{b}_1, \mathbf{b}_2, \ldots, \mathbf{b}_i \}) \setminus I(\{ \mathbf{b}_1, \mathbf{b}_2, \ldots, \mathbf{b}_{i+1} \}) $$ 
such that $ \deg(G) < i+1 $. The skew polynomial 
$$ F_{i+1} = F_i + \left( a_{i+1} - F_i(\mathbf{b}_{i+1}) \right) G(\mathbf{b}_{i+1})^{-1} G $$
satisfies that $ F_{i+1}(\mathbf{b}_j) = a_j $, for $ j = 1,2, \ldots, i+1 $, and $ \deg(F_{i+1}) < i+1 $. 
\end{enumerate}
\end{proof}

In the rest of the section, we derive some important consequences of this theorem. We start with the concept of dual P-bases.

\begin{definition} [\textbf{Dual P-bases}]
Given a finite P-basis $ \mathcal{B} = \{ \mathbf{b}_1, \mathbf{b}_2, \ldots, \mathbf{b}_M \} $ of a P-closed set $ \Omega \subseteq \mathbb{F}^n $, we say that a set of skew polynomials 
$$ \mathcal{B}^* = \{ F_1, F_2, \ldots, F_M \} \subseteq \mathbb{F}[\mathbf{x}; \sigma, \delta] $$
is a dual P-basis of $ \mathcal{B} $ if $ F_i(\mathbf{b}_j) = \delta_{i,j} $ for all $ i,j = 1,2, \ldots, M $.
\end{definition}

We have the following immediate consequence of Theorem \ref{th lagrange interpolation} on the existence and uniqueness of dual P-bases:

\begin{corollary} \label{corollary existence dual P-bases}
Any finite P-basis, with $ M $ elements, of a P-closed set $ \Omega $ admits a dual P-basis consisting of $ M $ skew polynomials of degree less than $ M $. Moreover, any two dual P-bases of the same P-basis define the same skew polynomial functions over $ \Omega $.
\end{corollary}

An important consequence of Theorem \ref{th lagrange interpolation} is the following result on the sizes of P-bases:

\begin{corollary} \label{corollary all P-bases same size}
Any two P-bases of a finitely generated P-closed set are finite and have the same number of elements.
\end{corollary}
\begin{proof}
Given a finite P-basis $ \mathcal{B} = \{ \mathbf{b}_1, \mathbf{b}_2, \ldots, \mathbf{b}_M \} $ of size $ M $ of a P-closed set $ \Omega $, we will show first that $ {\rm Im}(E^{\sigma, \delta}_{\Omega}) $ is a left vector subspace of $ \mathbb{F}^{\Omega} $ with basis 
\begin{equation}
\left\lbrace E^{\sigma, \delta}_{\Omega}(F_1), E^{\sigma, \delta}_{\Omega}(F_2), \ldots, E^{\sigma, \delta}_{\Omega}(F_M) \right\rbrace ,
\label{eq dual P-basis is basis}
\end{equation}
for any dual P-basis $ \{ F_1, F_2, \ldots, F_M \} $ of $ \mathcal{B} $. Assume that there exist $ \lambda_1, \lambda_2, \ldots, \lambda_M \in \mathbb{F} $ such that $ \sum_{i=1}^M \lambda_i E_\Omega^{\sigma, \delta}(F_i) = 0 $. Defining $ F = \sum_{i=1}^M \lambda_i F_i $, it follows that
$$ E_\Omega^{\sigma, \delta}(F) = E_\Omega^{\sigma, \delta} \left( \sum_{i=1}^M \lambda_i F_i \right) = \sum_{i=1}^M \lambda_i E_\Omega^{\sigma, \delta}(F_i) = 0, $$
thus $ \lambda_i = F(\mathbf{b}_i) = 0 $, for $ i = 1,2, \ldots, M $, and the set in (\ref{eq dual P-basis is basis}) is left linearly independent. Now, given $ F \in \mathbb{F}[\mathbf{x}; \sigma, \delta] $, define $ G = F - \sum_{i=1}^M F(\mathbf{b}_i) F_i $. By definition, we have that $ E_\mathcal{B}^{\sigma, \delta}(G) = 0 $. Therefore $ E_\Omega^{\sigma, \delta}(G) = 0 $ by Theorem \ref{th lagrange interpolation}, thus 
$$ E_\Omega^{\sigma, \delta}(F) = \sum_{i=1}^M F(\mathbf{b}_i) E_\Omega^{\sigma, \delta}(F_i), $$
and we conclude that the set in (\ref{eq dual P-basis is basis}) is a left basis of $ {\rm Im}(E^{\sigma, \delta}_{\Omega}) $.
 
In particular, $ \dim( {\rm Im} ( E^{\sigma, \delta}_{\Omega} ) ) = M $ and the result follows for finite P-bases, since $ \dim ( {\rm Im} ( E^{\sigma, \delta}_{\Omega} ) ) $ is independent of the choice of finite P-basis (since $ \mathbb{F} $ is a division ring).

Finally, if there exists an infinite P-basis $ \mathcal{B}^{\prime} $ of $ \Omega $, we may take a P-independent subset $ \mathcal{C} \subseteq \mathcal{B}^{\prime} $ of size $ M+1 $, define $ \Psi = \overline{\mathcal{C}} \subseteq \Omega $ and we would have that $ \dim( {\rm Im} ( E^{\sigma, \delta}_{\Psi} ) ) = M+1 $, and the canonical projection map
$$ \pi_\Psi : {\rm Im} ( E^{\sigma, \delta}_{\Omega} ) \longrightarrow {\rm Im} ( E^{\sigma, \delta}_{\Psi} ) $$
is onto. This is absurd (since $ \mathbb{F} $ is a division ring)  and the result follows. 
\end{proof}

We conclude with the following natural definition, which is motivated by the previous corollary. It is an extension of the case $ n = 1 $ given in \cite{lam, lam-leroy}.

\begin{definition} [\textbf{Rank of P-closed sets}]
Given a finitely generated P-closed set $ \Omega \subseteq \mathbb{F}^n $, we define its rank, denoted by $ {\rm Rk}(\Omega) $, as the size of any of its P-bases.
\end{definition}

\begin{remark}
If $ \Omega $ is a finitely generated P-closed set, then $ {\rm Rk}(\Omega) $ coincides with the rank of the matroid $ (\mathcal{P}(\Omega), \mathcal{I}_\Omega) $ from Proposition \ref{prop matroid} (see \cite[Sec. 1.3]{oxley}). Note that we make use of Corollary \ref{corollary all P-bases same size} to prove later in Lemma \ref{lemma adding one to P-independent} that $ (\mathcal{P}(\Omega), \mathcal{I}_\Omega) $ is indeed a matroid.
\end{remark}

\section{The image and kernel of the evaluation map} \label{sec image and ker}

In this section, we describe the left vector space of skew polynomial functions and, to that end, we obtain the dimensions and some left bases of the image and kernel of the evaluation map (Theorem \ref{th describing evaluation as vector space}). As a conclusion to the section, we also deduce a finite-dimensional left vector space description of quotients of a skew polynomial ring, which includes the minimal skew polynomial ring when $ \mathbb{F} $ is a finite field.

We need the following auxiliary lemmas. The first can be seen as a refinement of Item 1 in Theorem \ref{th lagrange interpolation}:

\begin{lemma} \label{lemma for skew vandermonde}
Let $ \Omega \subseteq \mathbb{F}^n $ be a finitely generated P-closed set and let $ \mathcal{G} \subseteq \Omega $. It holds that $ \Omega = \overline{\mathcal{G}} $ if, and only if,
$$ \dim( {\rm Im} ( E^{\sigma, \delta}_{\Omega} )) = \dim( {\rm Im} ( E^{\sigma, \delta}_{\mathcal{G}} )). $$
\end{lemma}
\begin{proof}
First, recall that the given dimensions are finite due to the proof of Corollary \ref{corollary all P-bases same size}. The direct implication is in essence Item 1 in Theorem \ref{th lagrange interpolation}. For the reversed implication, the equality on dimensions implies that the projection map $ {\rm Im} ( E^{\sigma, \delta}_{\Omega} ) \longrightarrow {\rm Im} ( E^{\sigma, \delta}_{\mathcal{G}} ) $ is a left vector space isomorphism. Thus $ I(\mathcal{G}) = I(\Omega) $, which implies that
$$ \overline{\mathcal{G}} = Z(I(\mathcal{G})) = Z(I(\Omega)) = \Omega. $$
\end{proof}

The next lemma is a further refinement of Proposition \ref{prop characterizations P-bases}:

\begin{lemma} \label{lemma adding one to P-independent}
If $ \mathcal{B} \subseteq \mathbb{F}^n $ is finite and P-independent, and $ \mathbf{a} \in \mathbb{F}^n \setminus \overline{\mathcal{B}} $, then $ \mathcal{B}^\prime = \mathcal{B} \cup \{ \mathbf{a} \} $ is P-independent. 

As a consequence, a finite subset $ \mathcal{B} \subseteq \mathbb{F}^n $ is a P-basis of a finitely generated P-closed set $ \mathcal{B} \subseteq \Omega \subseteq \mathbb{F}^n $ if, and only if, the following property holds: $ \mathcal{B} $ is P-independent, and if $ \mathcal{B} \subseteq \mathcal{G} \subseteq \Omega $ and $ \mathcal{G} $ is P-independent, then $ \mathcal{G} = \mathcal{B} $. That is, $ \mathcal{B} $ is a maximal P-independent set in $ \Omega $.
\end{lemma}
\begin{proof}
Since $ \mathbf{a} \notin \overline{\mathcal{B}} $, it holds that $ I(\mathcal{B}) \setminus I(\mathcal{B}^\prime) \neq \varnothing $, as in the proof of Lemma \ref{lemma previous to lagrange}. Thus 
$$ \dim({\rm Im}(E^{\sigma, \delta}_{\mathcal{B}^\prime})) \geq \dim({\rm Im}(E^{\sigma, \delta}_{\mathcal{B}})) + 1. $$
By the previous lemma and the proof of Corollary \ref{corollary all P-bases same size}, we conclude that $ {\rm Rk}(\overline{\mathcal{B}^\prime}) = {\rm Rk}(\overline{\mathcal{B}}) + 1 $. Again by Corollary \ref{corollary all P-bases same size} and its proof, we conclude that $ \mathcal{B}^\prime $ is a P-basis of $ \overline{\mathcal{B}^\prime} $ and, in particular, it is P-independent.
\end{proof}

Before giving the main result of this section, we need another consequence of Theorem \ref{th lagrange interpolation}, which will allow us to define the concepts of complementary P-closed sets and complementary P-bases:

\begin{corollary}
Let $ \Psi \subseteq \Omega \subseteq \mathbb{F}^n $ be P-closed sets. If $ \Omega $ is finitely generated, then so is $ \Psi $. Moreover, for any finite P-basis $ \mathcal{B} $ of $ \Psi $, there exists a finite P-independent set $ \mathcal{C} \subseteq \Omega $ such that $ \mathcal{B} \cap \mathcal{C} = \varnothing $ and $ \mathcal{B} \cup \mathcal{C} $ is a P-basis of $ \Omega $. In particular, if $ \Phi = \overline{\mathcal{C}} $, then
$$ {\rm Rk}(\Omega) = {\rm Rk}(\Psi) + {\rm Rk}(\Phi). $$
\end{corollary}
\begin{proof}
Assume that $ \Psi $ is not finitely generated. Using Lemma \ref{lemma adding one to P-independent}, we may construct iteratively a P-independent set $ \mathcal{D} \subseteq \Psi $ of size $ {\rm Rk}(\Omega) + 1 $. This is absurd by the same argument as in the proof of Corollary \ref{corollary all P-bases same size}.

Now, we may extend $ \mathcal{B} $ to a maximal P-independent subset of $ \Omega $ by adding iteratively to it elements $ \mathbf{c}_1, \mathbf{c}_2, \ldots, \mathbf{c}_N \in \Omega $, again by Lemma \ref{lemma adding one to P-independent}, which would be a P-basis of $ \Omega $ by maximality (again by Lemma \ref{lemma adding one to P-independent}). By defining $ \mathcal{C} = \{ \mathbf{c}_1, \mathbf{c}_2, \ldots, \mathbf{c}_N \} $, the rest of the claims in the corollary follow.
\end{proof}

\begin{definition} [\textbf{Complementary P-closed sets and P-bases}]
If $ \Psi \subseteq \Omega \subseteq \mathbb{F}^n $ are finitely generated P-closed sets and $ \mathcal{B} $ and $ \mathcal{C} $ are as in the previous Corollary, then we say that $ \Phi = \overline{\mathcal{C}} \subseteq \Omega $ is a complementary P-closed set of $ \Psi $ in $ \Omega $, and $ \mathcal{C} $ is a complementary P-basis of $ \mathcal{B} $ in $ \Omega $.
\end{definition}

We may now state and prove the second main result of the paper, which describes the image and kernel of $ E^{\sigma, \delta}_{\Omega} $ as left vector spaces over $ \mathbb{F} $ with some particular left bases. 

\begin{theorem} \label{th describing evaluation as vector space}
Given a finitely generated P-closed set $ \Omega \subseteq \mathbb{F}^n $ with finite P-basis $ \mathcal{B} $, we have that
\begin{enumerate}
\item
$ {\rm Im} ( E^{\sigma, \delta}_{\Omega} ) $ is a left vector space over $ \mathbb{F} $ of dimension $ M = {\rm Rk}(\Omega) $ with left basis 
$$ E^{\sigma,\delta}_{\Omega}(\mathcal{B}^*) = \left\lbrace E^{\sigma, \delta}_{\Omega}(F_1), E^{\sigma, \delta}_{\Omega}(F_2), \ldots, E^{\sigma, \delta}_{\Omega}(F_M) \right\rbrace , $$
where $ \mathcal{B}^* = \{ F_1, F_2, \ldots, F_M \} $ is a dual P-basis of $ \mathcal{B} $. Observe that, by Corollary \ref{corollary existence dual P-bases}, $ E^{\sigma,\delta}_{\Omega}(\mathcal{B}^*) $ depends only on $ \mathcal{B} $ and not on the choice of the dual P-basis.
\item
If $ \mathbb{F}^n $ is finitely generated as P-closed set, $ \mathcal{C} $ is a complementary P-basis of $ \mathcal{B} $ in $ \mathbb{F}^n $, and $ \mathcal{C}^* = \{ G_1, G_2, \ldots, G_N \} $ is a dual P-basis of $ \mathcal{C} $ that is part of a dual P-basis $ (\mathcal{B} \cup \mathcal{C})^* $ of $ \mathcal{B} \cup \mathcal{C} $, then 
$$ {\rm Ker} \left( E^{\sigma, \delta}_{\Omega} \right) = I(\mathbb{F}^n) \oplus \langle G_1, G_2, \ldots, G_N \rangle, $$
as left vector spaces over $ \mathbb{F} $, and $ G_1, G_2, \ldots, G_N $ are left linearly independent over $ \mathbb{F} $.
\end{enumerate}
\end{theorem}
\begin{proof}
The proof of Item 1 was given in the proof of Corollary \ref{corollary all P-bases same size}. Now we prove Item 2:

First, $ G_1, G_2, \ldots, G_N $ are left linearly independent over $ \mathbb{F} $ by Item 1, since so are their evaluations over $ \Phi = \overline{\mathcal{C}} $. 

Now we show that, if $ F \in I(\mathbb{F}^n) \cap \langle G_1, G_2, \ldots, G_N \rangle $, then $ F = 0 $. To that end, write $ F = \sum_{i=1}^N \lambda_i G_i $, for some $ \lambda_i \in \mathbb{F} $ and all $ i = 1,2, \ldots, N $. Since $ F \in I(\mathbb{F}^n) $, it holds that $ E_\Phi^{\sigma, \delta}(F) = 0 $, and since $ E^{\sigma, \delta}_{\Phi}(G_1), E^{\sigma, \delta}_{\Phi}(G_2), \ldots, E^{\sigma, \delta}_{\Phi}(G_M) $ are left linearly independent by Item 1, we conclude that $ F = 0 $.

Next let $ F \in {\rm Ker} ( E^{\sigma, \delta}_{\Omega} ) $. Then by Theorem \ref{th lagrange interpolation}, it holds that 
$$ F - \sum_{i=1}^N F(\mathbf{c}_i) G_i \in I(\mathbb{F}^n), $$
since this skew polynomial vanishes at $ \mathcal{B} \cup \mathcal{C} $, and this set is a P-basis of $ \mathbb{F}^n $. Hence $ F \in I(\mathbb{F}^n) \oplus \langle G_1, G_2, \ldots, G_N \rangle $. 

Conversely, let $ F \in I(\mathbb{F}^n) \oplus \langle G_1, G_2, \ldots, G_N \rangle $. By the assumptions, we have that $ F(\mathbf{b}) = 0 $, for all $ \mathbf{b} \in \mathcal{B} $. Hence $ F \in {\rm Ker}( E^{\sigma, \delta}_{\Omega} ) $ again by Theorem \ref{th lagrange interpolation}.
\end{proof}

We conclude with the following consequence, which describes the quotient left modules over the ideal associated to a finitely generated P-closed set. Such quotient left modules include the minimal skew polynomial ring if $ \mathbb{F}^n $ is finitely generated, which is the case if $ \mathbb{F} $ is finite.

\begin{corollary}
If $ \{ F_1, F_2, \ldots, F_M \} $ is a dual P-basis of a finitely generated P-closed set $ \Omega \subseteq \mathbb{F}^n $, then
$$ \mathbb{F}[\mathbf{x}; \sigma, \delta] / I(\Omega) \cong \langle F_1, F_2, \ldots, F_M \rangle $$
as left vector spaces, where the isomorphism is given by inverting the projection to the quotient ring. Moreover, $ F_1, F_2, \ldots, F_M  $ are left linearly independent, and hence
$$ \dim \left( \mathbb{F}[\mathbf{x}; \sigma, \delta] / I(\Omega) \right) = {\rm Rk}(\Omega). $$
In particular, the minimal skew polynomial ring $ \mathbb{F}[\mathbf{x}; \sigma, \delta] / I(\mathbb{F}^n) $ is a finite-dimensional left vector space over $ \mathbb{F} $ of dimension $ {\rm Rk}(\mathbb{F}^n) $ if $ \mathbb{F}^n $ is finitely generated.
\end{corollary}
\begin{proof}
Again, it follows directly from Item 1 in Theorem \ref{th describing evaluation as vector space} that $ F_1, F_2, \ldots, F_M $ are left linearly independent over $ \mathbb{F} $, since so are their evaluations over $ \Omega $.

Now consider the left linear projection map $ \rho : \langle F_1, F_2, \ldots, F_M \rangle \longrightarrow \mathbb{F}[\mathbf{x}; \sigma, \delta] / I(\Omega) $. To show that it is onto, it suffices to observe that, given $ F \in \mathbb{F}[\mathbf{x}; \sigma, \delta] $, it holds that
$$ \rho(F) = \rho \left( \sum_{i=1}^M F(\mathbf{b}_i) F_i \right), $$
by Item 1 in Theorem \ref{th lagrange interpolation}, where $ \mathcal{B} = \{ \mathbf{b}_1, \mathbf{b}_2, \ldots, \mathbf{b}_M \} $ is the P-basis of $ \Omega $ associated to $ \{ F_1, F_2, \ldots, $ $ F_M \} $.

Finally, the evaluation map $ \mathbb{F}[\mathbf{x}; \sigma, \delta] / I(\Omega) \longrightarrow {\rm Im} ( E^{\sigma, \delta}_{\Omega} ) $ is a left vector space isomorphism by definition, thus by Item 1 in Theorem \ref{th describing evaluation as vector space}, it holds that
$$ \dim ( \langle F_1, F_2, \ldots, F_M \rangle ) = M = \dim ( {\rm Im} ( E^{\sigma, \delta}_{\Omega} ) ) = \dim (\mathbb{F}[\mathbf{x}; \sigma, \delta] / I(\Omega)). $$ 
Hence $ \rho $ is a left vector space isomorphism, and we are done.
\end{proof}

As shown in Equation (\ref{eq min conventional skew pol ring for finite}), if $ \mathbb{F} $ is finite and has $ q $ elements, then 
$$ \dim(\mathbb{F}[\mathbf{x}] / I(\mathbb{F}^n)) = q^n = {\rm Rk}(\mathbb{F}^n), $$
since $ {\rm Rk}(\mathbb{F}^n) = \# \mathbb{F}^n = q^n $ in the conventional case. Hence the previous corollary extends this well-known result for finite fields.

\section{Skew Vandermonde matrices and how to find P-bases} \label{sec skew vandermonde}

In the univariate case ($ n = 1 $), Vandermonde matrices are a crucial tool to explicitly compute Lagrange interpolating polynomials.   The multivariate case works similarly, although only existence of interpolating skew polynomials may be derived, and not their uniqueness. This is due to the non-square form of multivariate Vandermonde matrices.

In this section, we extend the concept of \textit{skew Vandermonde matrix} from the univariate case in \cite{lam, lam-leroy} to the multivariate case. As applications and thanks to the recursive formula in Theorem \ref{th fundamental functions}, we show how to explicitly compute P-bases, dual P-bases and Lagrange interpolating skew polynomials over finitely generated P-closed sets.

The case $ n = 1 $ in the following definition was given in \cite[Eq. (4.1)]{lam-leroy}, and previously in \cite{lam}:

\begin{definition} [\textbf{Skew Vandermonde matrices}]
Let $ \mathcal{N} \subseteq \mathcal{M} $ be a finite set of monomials and let $ \mathcal{B} = \{ \mathbf{b}_1, \mathbf{b}_2, \ldots, \mathbf{b}_M \} \subseteq \mathbb{F}^n $. We define the corresponding $ (\sigma, \delta) $-skew Vandermonde matrix, denoted by $ V_{\mathcal{N}}^{\sigma, \delta}(\mathcal{B}) $, as the $ ( \# \mathcal{N} ) \times M $ matrix over $ \mathbb{F} $ whose rows are given by
$$ (N_\mathfrak{m}(\mathbf{b}_1), N_\mathfrak{m}(\mathbf{b}_2), \ldots, N_\mathfrak{m}(\mathbf{b}_M)) \in \mathbb{F}^M, $$
for all $ \mathfrak{m} \in \mathcal{N} $ (given certain ordering in $ \mathcal{M} $). If $ d $ is a positive integer, we define $ \mathcal{M}_d $ as the set of monomials of degree less than $ d $, and we denote
$$ V_{d}^{\sigma, \delta}(\mathcal{B}) = V_{\mathcal{M}_d}^{\sigma, \delta}(\mathcal{B}). $$
\end{definition}

An important consequence of Theorem \ref{th lagrange interpolation} is finding the rank and a P-basis of a given finitely generated P-closed set:

\begin{proposition} \label{prop rank vandermonde}
Given a finite set $ \mathcal{G} \subseteq \mathbb{F}^n $ with $ M $ elements, and $ \Omega = \overline{\mathcal{G}} $, it holds that
$$ {\rm Rk} \left( V_{M}^{\sigma, \delta}(\mathcal{G}) \right) = {\rm Rk}(\Omega). $$
Moreover, a subset $ \mathcal{B} \subseteq \mathcal{G} $ is a P-basis of $ \Omega $ if, and only if, $ \# \mathcal{B} = {\rm Rk}(\Omega) = {\rm Rk} ( V_{\# \mathcal{B}}^{\sigma, \delta}(\mathcal{B}) ) $.

Hence applying Gaussian elimination to the matrix $ V_{M}^{\sigma, \delta}(\mathcal{G}) $, we may find the rank of $ \Omega $ and at least one of its P-bases.
\end{proposition}
\begin{proof}
First, by Corollary \ref{corollary existence dual P-bases} and Theorem \ref{th describing evaluation as vector space}, it holds that $ {\rm Im} ( E_{\Omega}^{\sigma, \delta} ) $ is the left vector space generated by the evaluations $ (N_\mathfrak{m}(\mathbf{a}))_{\mathbf{a} \in \Omega} \in \mathbb{F}^\Omega $, for $ \mathfrak{m} \in \mathcal{M}_M $. By Lemma \ref{lemma for skew vandermonde}, to calculate $ \dim({\rm Im} ( E_{\Omega}^{\sigma, \delta} )) $, we may restrict such evaluations to points in $ \mathcal{G} $, and the first claim follows.

Now we prove the second claim. If $ \mathcal{B} $ is a P-basis of $ \Omega $, then $ \# \mathcal{B} = {\rm Rk}(\Omega) $ by definition, and $ {\rm Rk}(\Omega) = {\rm Rk} ( V_{\# \mathcal{B}}^{\sigma, \delta}(\mathcal{B}) ) $ by the first claim.

Conversely, if $ \# \mathcal{B} = {\rm Rk}(\Omega) = {\rm Rk} ( V_{\# \mathcal{B}}^{\sigma, \delta}(\mathcal{B}) ) $, then by Theorem \ref{th describing evaluation as vector space}, it holds that
$$ \dim( {\rm Im} ( E^{\sigma, \delta}_{\Omega} )) = {\rm Rk}(\Omega) = {\rm Rk} ( V_{\# \mathcal{B}}^{\sigma, \delta}(\mathcal{B}) ) \leq \dim( {\rm Im} ( E^{\sigma, \delta}_{\mathcal{B}} )). $$
Since the opposite inequality always holds, it follows from Lemma \ref{lemma for skew vandermonde} that $ \overline{\mathcal{B}} = \Omega $. Now, $ \mathcal{B} $ is a minimal set of P-generators of $ \Omega $, since $ \# \mathcal{B} = {\rm Rk}(\Omega) $, all minimal sets of P-generators are P-bases by Proposition \ref{prop characterizations P-bases} and all have the same size by Corollary \ref{corollary all P-bases same size}. Hence we conclude that $ \mathcal{B} $ is a P-basis of $ \Omega $.
\end{proof}

A classical way of stating the Lagrange interpolation theorem is as the invertibility of Vandermonde matrices. This result   is an immediate consequence of Theorem \ref{th lagrange interpolation}:

\begin{corollary}
Let $ \Omega \subseteq \mathbb{F}^n $ be a finitely generated P-closed set with P-basis $ \mathcal{B} = \{ \mathbf{b}_1, \mathbf{b}_2, \ldots, $ $ \mathbf{b}_M \} $. There exists a solution to the linear system
\begin{equation}
(F_\mathfrak{m})_{\mathfrak{m} \in \mathcal{M}_M} V_{M}^{\sigma, \delta}(\mathcal{B}) = (a_1, a_2, \ldots, a_M),
\label{eq system for interpolation}
\end{equation}
for any $ a_1, a_2, \ldots, a_M \in \mathbb{F} $ (that is, $ V_{M}^{\sigma, \delta}(\mathcal{B}) $ is left invertible). For any solution, the corresponding skew polynomial $ F = \sum_{\mathfrak{m} \in \mathcal{M}_M} F_\mathfrak{m} \mathfrak{m} $ satisfies that $ F(\mathbf{b}_i) = a_i $, for $ i = 1,2, \ldots, M $, and $ \deg(F) < M $. 
\end{corollary}

Another important immediate consequence is the following:

\begin{corollary}
Given a P-basis $ \mathcal{B} $, with $ M $ elements, of a P-closed set, one can obtain a dual P-basis of $ \mathcal{B} $, consisting of skew polynomials of degree less than $ M $, by solving $ M $ systems of $ M $ linear equations whose coefficients are taken from left linearly independent rows in $ V_{M}^{\sigma, \delta}(\mathcal{B}) $.
\end{corollary}

In conclusion, to find a P-basis of a P-closed set $ \Omega \subseteq \mathbb{F}^n $ with $ M = {\rm Rk}(\Omega) $ and generated by a finite set $ \mathcal{G} $, we need to find $ M $ linearly independent columns in $ V_{M}^{\sigma, \delta}(\mathcal{G}) $. Using Gaussian elimination, such method has exponential complexity in $ M $ if $ n > 1 $, since the number of rows in $ V_M^{\sigma, \delta}(\mathcal{G}) $ is $ \# \mathcal{M}_M $, which is exponential in $ M $.

Fortunately, if we are given or have precomputed a P-basis of $ \Omega $, we may find Lagrange interpolating skew polynomials over $ \Omega $ with complexity $ \mathcal{O}(M^3) $, and find a dual P-basis with complexity $ \mathcal{O}(M^4) $.

\section{Conclusion and open problems}

In this paper, we have introduced free multivariate skew polynomials with coefficients over rings, although we have focused on division rings. We have given a natural definition of evaluation (Definition \ref{def standard evaluation}), which extends the univariate case studied in \cite{lam, algebraic-conjugacy, lam-leroy}, and we have obtained a product rule (Theorem \ref{th product rule}). With these notions and assumptions, we were able to define general nonfree multivariate skew polynomial rings (Definition \ref{def skew polynomial rings}), where evaluation is still natural. We have described (by giving dimensions and left bases) in Theorem \ref{th describing evaluation as vector space} the left vector spaces of functions defined by multivariate skew polynomials, when defined over a finitely generated P-closed set (set of zeros). This has been done thanks to a Lagrange-type interpolation theorem (Theorem \ref{th lagrange interpolation}). The following problems are left open:
\begin{enumerate}
\item
Find explicit descriptions of general multivariate skew polynomial rings. In other words, find explicit descriptions of matrix morphisms, vector derivations and two-sided ideals contained in $ I(\mathbb{F}^n) $. 
\item
The previous item is particularly interesting in the case of finite fields, where the minimal skew polynomial ring is generated by a finite collection of skew polynomials. 
\item
Although we have given computational methods to find ranks, P-bases, dual P-bases and Lagrange interpolating skew polynomials, it would be interesting to obtain explicit formulas for such objects. Algorithms for finding P-bases with polynomial complexity are also interesting, as well as reducing the complexity of finding Lagrange interpolating skew polynomials.
\item
Investigate how to perform Euclidean-type divisions over multivariate skew polynomial rings, which would extend Lemma \ref{th euclidean division}. A notion of Gr{\"o}bner basis may be possible and useful in this context.
\end{enumerate}

\section*{Acknowledgement}

The authors wish to thank the anonymous reviewer for their very helpful comments. The first author gratefully acknowledges the support from The Independent Research Fund Denmark (Grant No. DFF-4002-00367, Grant No. DFF-5137-00076B ``EliteForsk-Rejsestipendium'', and Grant No. DFF-7027-00053B).

\small


\end{document}